\UseRawInputEncoding 
\documentclass[12pt]{amsart}       
\usepackage{txfonts}
\usepackage{algorithm}
\usepackage{algorithmic}
\usepackage{amssymb}
\usepackage{mathrsfs}  
\usepackage{amsmath}   
\usepackage{eucal}
\usepackage{graphicx}
\usepackage{amsmath}
\usepackage{amscd}
\usepackage[all]{xy}           
\usepackage{tikz}
\usepackage{tikz-qtree}
\usepackage{amsfonts,latexsym}
\usepackage{xspace}
\usepackage{epsfig}
\usepackage{float}
\usepackage{color}
\usepackage{fancybox}
\usepackage{colordvi}
\usepackage{multicol}
\usepackage{colordvi}
\usepackage{wasysym}
\usepackage{makecell} 
\usepackage[numbers,sort&compress]{natbib}
\usepackage[active]{srcltx} 
\usepackage{CJK}
\ifpdf
  \usepackage[colorlinks,final,backref=page,hyperindex]{hyperref}
\else
  \usepackage[colorlinks,final,backref=page,hyperindex,hypertex]{hyperref}
\fi

\newcommand{\nc}{\newcommand}
\newcommand{\delete}[1]{}
\nc{\bfk}{{\bf k}}

\nc{\mlabel}[1]{\label{#1}}  
\nc{\mcite}[1]{\cite{#1}}  
\nc{\mref}[1]{\ref{#1}}  
\nc{\mbibitem}[1]{\bibitem{#1}} 

\delete{
\nc{\mlabel}[1]{\label{#1}  
{\hfill \hspace{1cm}{\small\tt{{\ }\hfill(#1)}}}}
\nc{\mcite}[1]{\cite{#1}{\small{\tt{{\ }(#1)}}}}  
\nc{\mref}[1]{\ref{#1}{{\tt{{\ }(#1)}}}}  
\nc{\mbibitem}[1]{\bibitem[\bf #1]{#1}} 
}

\newtheorem{theorem}{Theorem}[section]

\newtheorem{lemma}[theorem]{Lemma}

\newtheorem{corollary}[theorem]{Corollary}
\theoremstyle{definition}
\newtheorem{definition}[theorem]{Definition}

\newtheorem{claim}{Claim}[section]

\newtheorem{tempex}[theorem]{Example}
\newtheorem{tempexs}[theorem]{Examples}
\newtheorem{temprmk}[theorem]{Remark}
\newtheorem{tempexer}{Exercise}[section]

\nc{\tred}[1]{\textcolor{red}{#1}} \nc{\tgreen}[1]{\textcolor{green}{#1}}
\nc{\tblue}[1]{\textcolor{blue}{#1}} \nc{\tpurple}[1]{\textcolor{purple}{#1}}

\nc{\hu}[1]{\tpurple{\underline{Hu:}#1 }}
\nc{\xing}[1]{\tblue{\underline{Xing:}#1 }}

\nc{\GS}{Gr\"obner-Shirshov\xspace}
\nc{\gsb}{Gr\"{o}bner-Shirshov basis\xspace}
\nc{\gsbs}{Gr\"{o}bner-Shirshov bases\xspace}
\nc\olie{operated Lie algebra\xspace}
\nc\olies{operated Lie algebras\xspace}
\nc\bfone{\mathbf{1}}

\nc\nas[1]{{#1}^\ast}
\nc\alsw[1]{{\rm ALSW}(#1)}
\nc\nlsw[1]{{\rm NLSW}(#1)}
\nc\clie[1]{[#1]}
\nc\lbar[1]{\overline{#1}}
\nc\suba[1]{|_{#1}}
\nc\coplie[1]{\bfk\nlsbw{#1}}
\nc\coplieo[2]{\bfk{\rm NLSBW}_{#2}(#1)}
\nc\lb[1]{\left[#1\right]}\nc\dt[1]{{t^{(#1)}}}
\nc{\Irr}{\mathrm{Irr}}
\nc\blw[1]{\lfloor#1\rfloor}
\nc\plie[1]{\mathfrak{S}(#1)}
\nc\plien[1]{\mathcal{N}(#1)}
\nc{\lc}{\lfloor} \nc{\rc}{\rfloor}
\nc\Id{\rm Id}\nc\sopm[1]{\mathfrak{S}^\star(#1)}

\nc\ordc{>_{{\rm Dl}}} \nc\ordqc{\geq_{{\rm Dl}}}
\nc\ord{>_{{\rm IM }}}\nc\ordq{\geq _{\rm IM}}
\nc\ordd{>_{{\rm IM}}}\nc\ordqd{\geq_{{\rm IM}}}
\nc\ordb{>_{{\rm IM}}}\nc\ordqb{\geq_{{\rm IM}}}
\nc\alsbw[1]{{\rm ALSBW}_{\ordq}(#1)} \nc\nlsbw[1]{{\rm NLSBW}_{\ordq}(#1)}
\nc\alsbwo[2]{{\rm ALSBW}_{#2}(#1)} \nc\nlsbwo[2]{{\rm NLSBW}_{#2}(#1)}\nc\bws[1]{{\lfloor#1\rfloor}}\nc\oplie{{\rm OLie}(X)}
\nc{\dep}{{\rm dep}}
\nc\ordt{\geq_{\rm dt}}
\nc\ordtl{>_{\rm dt}}

\begin{document}

\title[Immanantal polynomials of the linear combination matrices of  graphs]{Immanantal polynomials of the linear combination matrices of  graphs}

\author{Xiangshuai Dong, Tingzeng Wu$^{*}$}\thanks{*Corresponding author}

\address{School of Mathematical Sciences, Xiamen University, Xiamen, Fujian 361005, China}
\email{a3566293588@163.com}

\address{School of Mathematics and Statistics, Qinghai Minzu University, Xining, Qinghai 810007, China}
\email{mathtzwu@163.com}

\date{\today}

\begin{abstract}
The immanant is a generalized matrix function associated with irreducible characters of the symmetric group. The  immanantal polynomial of matrix $M$  is defined as the immanant of matrix $(xI - M)$, where $I$ is the  identity matrix. The characteristic polynomials and permanental polynomials of graph matrices (adjacency matrix, Laplacian matrix, signless Laplacian matrix and  $A_{\alpha}(G)$ matrix)  have been extensively studied, yielding fascinating results. As a unification of the characteristic polynomial and permanental polynomial, the  immanantal polynomial has also attracted attention and led to several important findings. In this paper, we focus on the study of  immanantal polynomials for linear combination matrices composed of the degree matrix and adjacency matrix of  a graph. First, applying the concept of vertex orientation for general graphs, we provide a combinatorial interpretation of the coefficients of the  immanantal polynomials for the linear combination matrices of   graphs, and we also characterize the bounds of these coefficients. These bounds implicitly encompass the existing results of Chan and Lam on trees and bipartite graphs. Furthermore, we give a solution to  the open problem posed by Merris. Second, we characterize the first six coefficients of the hook immanantal polynomial. And the necessary and sufficient condition under which the linear combination matrices of two regular graphs have the same hook immanantal polynomial is proved.  Third,
we generalize the  Frobenius--K\"onig theorem and the Laplace expansion theorem to immanants.
Using these two theorems, we show that the star degree of a graph is always a lower bound for the multiplicity of a certain root of the immanantal polynomial of its linear combination matrix.
Finally,  we derive formulas for the first six coefficients of the hook immanantal polynomial for several important graph matrices.

\end{abstract}

\makeatletter
\@namedef{subjclassname@2020}{\textup{2020} Mathematics Subject Classification}
\makeatother
\subjclass[2020]{
     05B20,
    05E05,
    05C31, 05C50, 15A15.
}

\keywords{Immanant; immanantal polynomial; Hook immanantal polynomial; Linear combination  matrix; Graph matrix; Coefficient; Root}

\maketitle

\tableofcontents

\setcounter{section}{0}

\allowdisplaybreaks

\section{Introduction}
Let $\chi_\lambda$ be an irreducible character of the symmetric group $\mathcal{S}_n$, indexed by a partition  $\lambda$ of $n$. For an $n\times n$ matrix $M=[m_{ij}]$, the \emph{immanant} of $M$ associated with $\chi_\lambda$ is defined as
\begin{eqnarray}\label{equ1.1}
\operatorname{Imm}_\lambda (M) = \sum_{\sigma \in \mathcal{S}_{n}} \chi_{\lambda}(\sigma) \prod_{i=1}^n m_{i\sigma(i)}.
\end{eqnarray}
When $\lambda = (k, 1^{n-k})$, $\operatorname{Imm}_{(k, 1^{n-k})}(\cdot)$ is called the \emph{hook immanant}. B\"{u}rgisser \cite{bur} proved that the computation of hook immanants with polynomially growing width is both \#P-complete and VNP-complete.

Schur \cite{sch} first considered immanants as generalized matrix functions. The \emph{determinant} and \emph{permanent} are immanants corresponding to the sign character and the trivial character, respectively, i.e., $\operatorname{Imm}_{(1^n)}(\cdot)$ and $\operatorname{Imm}_{(n)}(\cdot)$. Littlewood \cite{lit} pointed out that the immanant of a matrix is a multilinear function interpolating between the determinant and the permanent. Combinatorial and algebraic properties of immanants can be found in \cite{bol, gou2, hai, li1, mar, ste}.

Let $I_n$ denote the $n\times n$ identity matrix. For an $n\times n$ matrix $M$ and an irreducible character $\chi_\lambda$, the \emph{immanantal polynomial} of $M$ associated with $\chi_\lambda$ is defined as
\begin{eqnarray}\label{equ1.2}
\operatorname{Imm}_{\lambda}(xI_n - M) = \sum_{r=0}^{n} (-1)^{r} c_{\lambda, r}(M) \, x^{\,n-r}.
\end{eqnarray}
In particular,
\begin{itemize}
\item when $\lambda = (k, 1^{\,n-k})$, $\operatorname{Imm}_{(k, 1^{\,n-k})}(xI_n - M)$ is called the \emph{hook immanantal polynomial} of $M$,
\item when $\lambda = (1^{n})$, $\operatorname{Imm}_{(1^{n})}(xI_n - M)$ is the \emph{characteristic polynomial} of $M$,
\item when $\lambda = (2,1^{\,n-2})$, $\operatorname{Imm}_{(2,1^{\,n-2})}(xI_n - M)$ is the \emph{second immanantal polynomial} of $M$,
\item when $\lambda = (n)$, $\operatorname{Imm}_{(n)}(xI_n - M)$ is the \emph{permanent polynomial} of $M$.
\end{itemize}
For studies on immanantal polynomials of general matrices, see \cite{bro, mer4}.

Let $G = (V(G), E(G))$ be a simple graph with vertex set $V(G) = \{v_1, v_2, \ldots, v_n\}$ and edge set $E(G)$. Denote by $d_{G}(v_i)$ (simply $d(v_i)$) the degree of vertex $v_i$. The degree matrix of $G$ is $D(G) = \mathrm{diag}(d(v_1), d(v_2), \ldots, d(v_n))$. The adjacency matrix $A(G) = [a_{ij}]$ is an $n\times n$ $(0,1)$-matrix with $a_{ij}=1$ if $v_i v_j \in E(G)$ and $0$ otherwise. The matrix $L(G) = D(G) - A(G)$ is the Laplacian matrix of $G$, and $L(G) = D(G)+A(G)$ is the signless Laplacian matrix. The $A_{\alpha}(G)$ matrix is defined as $A_{\alpha}(G) = \alpha D(G) + (1-\alpha) A(G)$ for $\alpha \in [0,1]$.

In algebraic and combinatorics, the study on immanantal polynomials of graph matrices is a long-standing and classic problem. Merris \cite{bot} proved that almost all trees have complete immanantal polynomials. Yu and Qu \cite{yu} derived explicit expressions for the immanantal polynomials of the adjacency, Laplacian and signless Laplacian matrices of  graphs. Cash \cite{cas} extended Sachs' theorem to immanantal polynomials and applied it to chemical graphs such as hydrocarbon structures. Moreover, immanants have important applications in various theoretical and applied sciences \cite{deg, dia, rho, sta}. Due to the computational complexity of immanants of graph matrices, relatively few results have been obtained. However, fruitful results exist for immanants under special partitions of graph matrices. Undoubtedly, the most abundant results are for determinants and permanents of graph matrices \cite{agr, bro, cvet, gly, god, liw, min, mow, van}. Merris \cite{mer4} studied the second immanantal polynomials of the Laplacian matrices of  graphs and proposed open problems concerning the bounds of its coefficients. Chan and Lam \cite{chan2, chan5} characterized the bounds for the coefficients of the immanantal polynomials of the Laplacian matrices of trees, partially solving the problem. Meanwhile, Merris \cite{mer4} also raised an open problem about the multiplicity of the root $1$ of this polynomial, which was later solved by Wu et al. \cite{wu2}. Chan and Lam \cite{chan3} established a relationship between the third immanantal polynomial of the Laplacian matrix of a tree and its Wiener index. A natural question is:
from the properties of immanant functions of graph matrices under special partitions (such as determinant and permanent), deduce the common properties that hold  under all partitions.  In this paper, we investigate this problem.

This paper mainly studies the linear combination matrix $\beta D(G) + \gamma A(G)$ of the degree matrix $D(G)$ and the adjacency matrix $A(G)$ of graph $G$, where $\beta$ and $\gamma$ are real numbers. Section $2$ presents some definitions and lemmas. We introduce the concept of vertex orientation of the general graph and some important graph parameters. In Section $3$, using  the tool of vertex orientation,  we give a combinatorial interpretation of the coefficients of the immanantal polynomial of the linear combination matrix for general graphs. Furthermore, we characterize the upper and lower bounds for these coefficients. In addition, we determine the bounds for the coefficients of the immanantal polynomial of the linear combination matrix for trees and bipartite graphs, thereby generalizing the results of Chan and Lam.
In Section $4$, we focus on the hook immanantal polynomial of the linear combination matrix for general graphs and give combinatorial interpretations of its first six coefficients. Using these coefficients, we prove the necessary and sufficient condition for two regular graphs to have equal hook immanantal polynomials of their linear combination matrices. Section $5$ investigates the roots of the immanantal polynomial of the linear combination matrix. we generalize the  Frobenius--K\"onig theorem and the Laplace expansion theorem to immanants. Applying these theorems, we prove that the star degree of a graph always provides a lower bound for the multiplicity of the root $\beta$ in the immanantal polynomial of its linear combination matrix. In Section $6$, we further discuss Theorem \ref{theorem4.1}. We give the formulas for the first six coefficients of the hook immanantal polynomial for several important classes of graph matrices. Additionally,  the classic known results of previous researchers can be derived.

\section{Basic preliminaries}
In this section, we present some definitions and lemmas that will be used in the proofs of our results.

Let $G$ be a simple graph with vertex set $V(G) = \{v_1, v_2, \ldots, v_n\}$. The graph obtained by deleting the edge $e$ from $G$ is denoted by $ G - e $. Denote by $\mathfrak{D}(G)$ the degree sequence of $G$, i.e., $\mathfrak{D}(G) = (d(v_{1}), d(v_{2}), \ldots, d(v_{n}))$. Suppose $\mathscr{C}_{l}(G)$ and $\mathscr{M}_{l}(G)$ denote the sets of cycles of length $l$ and $l$-matchings in $G$, respectively. Let $\mathscr{T}_{j}(G)$ denote the sum of degrees of the three vertices on the $j$th triangle in $\mathscr{C}_{3}(G)$. If $C =v_{i_{1}}v_{i_{2}}\cdots v_{i_{l}}v_{i_{1}}\in \mathcal{C}_{l}(G)$, let $ \mathfrak{D}_C(G)$ be the degree sequence obtained from $\mathfrak{D}(G)$ by deleting $d(v_{i_{1}})$, $d(v_{i_{2}})$,\ldots, $d(v_{i_{l}})$. If $\mathfrak{M} =\{v_{i_{1}}v_{i_{1}'}, v_{i_{2}}v_{i_{2}'},\ldots,v_{i_{l}}v_{i_{l}'}\}\in \mathscr{M}_{l}(G)$, let $ \mathfrak{D}_\mathfrak{M}(G)$ be the degree sequence obtained from $\mathfrak{D}(G)$ by deleting $d(v_{i_{1}})$, $d(v_{i_{1}'})$, $d(v_{i_{2}})$, $d(v_{i_{2}'})$,\ldots,$d(v_{i_{l}})$, $d(v_{i_{l}'})$.
Denote by $F_r$ the $r$th {\em elementary symmetric function}, and define $F_r(G) = F_r(\mathfrak{D}(G))$. For $r \geq l$, we define
\begin{eqnarray*}
\mathcal{C}_r^{l}(G) = \sum_{C \in \mathscr{C}_{l}(G)} F_{r-l}(\mathfrak{D}_C(G))
\end{eqnarray*}
and
\begin{eqnarray*}
\mathcal{M}_r^{l}(G) = \sum_{\mathfrak{M} \in \mathscr{M}_{l}(G)} F_{r-2l}(\mathfrak{D}_\mathfrak{M}(G)).
\end{eqnarray*}
A {\em pendant star} of  graph is defined as a maximal subgraph consisting of pendant edges, all incident to the same vertex (referred to as the center of the pendant star). A {\em pendant vertex} is a vertex with degree $1$. The {\em  degree} of a pendant star is defined as the number of its pendant vertices minus one. The {\em star degree} of  graph is defined as the sum of the degrees of all its pendant stars if such stars exist, otherwise, it is zero.  Let $C_{n}$, $P_{n}$, $S_{n}$ and $K_{n}$ denote the cycle, path, star, and complete graph, respectively, each having $n$ vertices.

Chan and Lam \cite{chan2,chan4} introduced the notion of the {\em vertex orientations} of trees and bipartite graphs. We extend this concept to general graphs. Let $G$ be a graph, and let $B$ is a subset of $V(G)$ with $r$ elements (or simply $r$-subset $B$  of $V(G)$).  Assume that  $G[B]$ be the subgraph induced on $B$. For each vertex $v$ in $B$, we assign an arrow pointing away from the vertex along one of its adjacent edges. There are $d_{G}(v)$ ways to do this. We call such an assignment for all vertices in $B$ a {\em$(B)$-vertex orientation}. In a $(B)$-vertex orientation, if there are $r_1$ edges with $1$ arrow, $r_2$ edges with 2 arrows, $r_3$ directed $3$-cycles, $r_4$ directed $4$-cycles, and so on, then we say the $(B)$-vertex orientation is of type
$(\nu,B)=(1^{r_{1}},2^{r_{2}},3^{r_{3}},4^{r_{4}},\dots)$, where $r_{1}+2r_{2}+3r_{3}+4r_{4}+\dots=n$. Define $a_G(\nu,r)$ to be the number of $(B)$-vertex orientation of type  $(\nu,B)$.    For convenience, let $\mathcal{P}(n)$ be the set of partitions of $n$, and let $\mathcal{PB}(n)$  be the set of all partitions of $n$ obtained from $(B)$-vertex orientation in $G$. We write $\lambda \vdash n$ if $\lambda$ is a partition of $n$.

Given a positive integer $n$, let $N_{r} = \{\{i_{1}, i_{2}, \ldots, i_{r}\}|1\leq i_{1}\leq i_{2}\leq \ldots\leq i_{r}\leq n\}$. For $I\in N_{r}$, the subgroup $(\mathcal{S}_n)_I $ of the symmetric group $ \mathcal{S}_n $ is defined as
\begin{eqnarray*}
(\mathcal{S}_n)_I = \left\{ \sigma \in \mathcal{S}_n \mid \sigma(i) = i \text{ for all } i \in \{1, \dots, n\} \setminus I \right\}.
\end{eqnarray*}
In particular, for  a subset $B=\{v_{i_{1}}, v_{i_{2}}, \ldots, v_{i_{r}}\}$ of   $V(G)=\{v_{1},v_{2},\ldots,v_{n}\}$ with an ordered sequence of indices $1 \le i_1 < i_2 < \cdots < i_r \le n$, we define the subgroup
\begin{eqnarray*}
(\mathcal{S}_n)_B = \left\{ \sigma \in \mathcal{S}_n \mid \sigma(i) = i \text{ for all } v_{i} \in V(G) \setminus B \right\}.
\end{eqnarray*}
Now, the character value of a permutation $\sigma\in (\mathcal{S}_{n})_{B}$ depends only on its cycle type, denoted by type $(\sigma)$. So we will write \(\chi_\lambda(\mu)\) instead of \(\chi_\lambda(\sigma)\) where $ \mu =$ type $(\sigma)$.  We consider permutations with cycle type $\mu = (1^{t_1}, 2^{t_2}, 3^{t_3}, 4^{t_4}, \ldots)$. Here $l^{t_l}$ indicates that there are $t_l$ $l$-cycles.

\begin{lemma}(Yu and Qu, \cite{yu})\label{lemm2.1}
Let $M = [m_{ij}]$ be an $n \times n$ matrix. Then
\begin{eqnarray*}
c_{\lambda,r}(M)=\sum\limits_{I\in N_{k}}\sum\limits_{\sigma\in(\mathcal{S}_n)_I }\chi_\lambda(\sigma)\prod\limits_{i\in I}m_{i\sigma(i)}.
\end{eqnarray*}
\end{lemma}

\begin{lemma}(Marcus and Nikolai, \cite{mar2})\label{lem3.6}
Let $M$, $N$ and $M-N$ be three positive semidefinite Hermitian matrices. Then
$$\operatorname{Imm}_{\lambda}(M)\geq \operatorname{Imm}_{\lambda}(N).$$
\end{lemma}

\begin{lemma}(Dong et al., \cite{don})\label{lem3.7}
Let $G$ be a graph. If $\beta>0$ and $\gamma\leq|\beta|$, then $\beta D(G) + \gamma A(G)$ is a positive semidefinite Hermitian matrix.
\end{lemma}

\begin{lemma}(Merris, \cite{mer1})\label{pro3.2}
Let $M_1$ and $M_2$ be two $n \times n$ matrices of the same structure. Then there exists a permutation matrix $N$ such that
\begin{eqnarray*}
\operatorname{Imm}_{\lambda}(M_1) = \operatorname{Imm}_{\lambda}(N^{-1}M_2N).
\end{eqnarray*}
\end{lemma}

\begin{lemma}\label{lem3.3}
Let $G$ be a graph. Then
\begin{eqnarray}\label{equ3.4}
a_G(\nu,r)\geq a_{G-e}(\nu,r).
\end{eqnarray}
\end{lemma}
\begin{proof}
Since any $(B)$-vertex orientation of type $(\nu,B)$ in $G-e$ is also a $(B)$-vertex orientation of type $(\nu,B)$ in $G$, the result in Lemma \ref{lem3.3} can be directly obtained.
\end{proof}

\begin{lemma}(Chan and Lam, \cite{chan5})\label{lem3.5}
Let $T$ be a tree with $n$ vertices. Then
$$a_{S_{n}}(\nu,r)\leq a_T(\nu,r)\leq a_{p_{n}}(\nu,r).$$
\end{lemma}

\begin{lemma}(Chan and Lam, \cite{chan5})\label{lem2.11}
Let $K_{p,q}$ denote the complete bipartite graph with $(p,q)$-bipartition, where $p\geq q$. Then
$$a_{K_{p,q}}(\nu)> a_{K_{p+1,q-1}}(\nu).$$
\end{lemma}

\begin{lemma}\label{lem2.6}
Let $K_{p,q}$ denote the complete bipartite graph with $(p,q)$-bipartition, where $p\geq q$. Then
\begin{eqnarray*}
a_{K_{p,q}}(\nu,r)> a_{K_{p+1,q-1}}(\nu,r).
\end{eqnarray*}
\end{lemma}
\begin{proof}
From Lemma \ref{lem2.11} and the definition of $(B)$-vertex orientation, we can directly obtain the result in Lemma \ref{lem2.6}.
\end{proof}

 \begin{lemma}(Sagan, \cite{sag})\label{the5.1}
Let $\lambda$ and $\alpha = (\alpha_1, \ldots, \alpha_i)$ be two partitions. Then
\begin{eqnarray*}
\chi_{\lambda}(\alpha) = \sum_{\xi} (-1)^{h(\xi)}\chi_{\lambda/\xi}(\alpha/\alpha_1),
\end{eqnarray*}
where the sum runs over all rim hooks $\xi$ of length $\alpha_1$ in $\lambda$, and $h(\xi) = (\text{the number of rows of } \xi) - 1$.
\end{lemma}

\begin{lemma}(Sagn, \cite{sag})\label{lemma2.13}
Suppose $\sigma_1 \sigma_2\in \mathcal{S}_{r}\times \mathcal{S}_{n-r}\subseteq \mathcal{S}_{n}$, $\lambda \vdash n$, $\mu \vdash r$ and $\nu \vdash n-r$.
Then
 \begin{eqnarray*}
\chi_{\lambda}(\sigma_1 \sigma_2)
&=& \sum_{\substack{\mu \vdash r \\ \nu \vdash n-r}}
c_{\mu, \nu}^\lambda
\; \chi_{\mu}(\sigma_1) \, \chi_{\nu}(\sigma_2),
\end{eqnarray*}
where $c_{\mu, \nu}^\lambda$ are Littlewood--Richardson coefficients.
\end{lemma}

\begin{lemma}(James and Kerber, \cite{jam})\label{lem5.1}
Let $\chi_{(k, 1^{n-k})}$ be the irreducible character of the symmetric group $\mathcal{S}_{n}$ indexed by the partition $(k, 1^{n-k})$. Let $\text{id}$ denote the identity permutation of $\mathcal{S}_{n}$. Then
\begin{eqnarray*}
\chi_{(k, 1^{n-k})}(\text{id})= \binom{n-1}{k-1}.
\end{eqnarray*}
\end{lemma}

\begin{lemma}\label{lem5.2}
Let $\sigma$ be a permutation in the symmetric group $\mathcal{S}_{n}$ with cycle type $(\sigma)=(l,1^{n-l})$. Then
\begin{eqnarray*}
\chi_{(k, 1^{n-k})}(\sigma)=\binom{n-l-1}{k-l-1}+(-1)^{l-1}\binom{n-l-1}{k-1}.
\end{eqnarray*}
\end{lemma}
\begin{proof}
By Lemma \ref{the5.1} and Lemma \ref{lem5.1}, we obtain
\begin{eqnarray*}
\chi_{(k, 1^{n-k})}(l,1^{n-l})&=&\chi_{(k-l, 1^{n-k})}(1^{n-l})+(-1)^{l-1}\chi_{(k, 1^{n-k-l})}(1^{n-l})\\
&=&\binom{n-l-1}{k-l-1}+(-1)^{l-1}\binom{n-l-1}{k-1}.
\end{eqnarray*}
The proof of Lemma \ref{lem5.2} is complete.
\end{proof}

\section{Bounds on coefficients}

In this section, we characterize the explicit formulas for the coefficients of the immanantal polynomials of the linear combination matrices of  graphs by applying $(B)$-vertex orientation. And  the upper and lower bounds for these coefficients is determined. Moreover, we determine the bounds for the coefficients of the immanantal polynomials of the linear combination matrices for trees and bipartite graphs. Meanwhile, we also deduce the existing results of Chan and Lam on trees and bipartite graphs, and we solve the open problem proposed by Merris.

\subsection{Bounds for the coefficients of the immanantal polynomials of the linear combination matrices of  graphs}

Prior to the introduction of our main results, we establish several lemmas.

\begin{lemma}\label{lemma3.1}
 Let $G$ be a graph with $n$ vertices, and let the matrix $\beta D(G) + \gamma A(G)=[h_{ij}]$, where  $\beta$ and $\gamma$ be two non-zero real numbers. Then
\begin{eqnarray*}
c_{\lambda,r}(\beta D(G) + \gamma A(G))
&=&\sum\limits_{B:|B|=r,B\subseteq V(G)}\sum_{\mu\in\mathcal{P}(n)} \chi_\lambda(\mu) \sum_{\sigma\in(\mathcal{S}_{n})_{B}\atop type(\sigma)=\mu} \prod_{i=1}^n h_{i\sigma(i)}\\
&=&\sum\limits_{B:|B|=r,B\subseteq V(G)}\sum_{\mu\in\mathcal{P}(n)} \chi_\lambda(\mu) \sum_{\sigma\in(\mathcal{S}_{n})_{B}\atop type(\sigma)=\mu}\beta^{\mathcal{F}(\sigma)}\gamma^{n-\mathcal{F}(\sigma)}\prod\limits_{v_{i}\in B,\sigma(i)=i}d(v_{i}),
\end{eqnarray*}
where $\mu$ is a conjugate class of the subgroup  $(\mathcal{S}_{n})_{B}$, and $\mathcal{F}(\sigma)$ is the number of fixed points in $\sigma$.
\end{lemma}

\begin{proof}
Suppose that $G$ is a graph with vertices $v_{1}, v_{2}, \ldots, v_{n}$. Index the rows and columns of the matrix $\beta D(G) + \gamma A(G)$ by $v_{1}, v_{2}, \ldots, v_{n}$. For convenience, we use $H(G)=[h_{ij}]$ to denote  $\beta D(G) + \gamma A(G)$.  By Lemma \ref{lemm2.1}, we have

\begin{eqnarray*}
c_{\lambda,r}(H(G)) &=& \sum_{B: |B|=r, B \subseteq V(G)} \sum_{\sigma \in (\mathcal{S}_n)_B} \chi_\lambda(\sigma) \prod_{v_{i} \in B} h_{i\sigma(i)}\\
&=&\sum_{B: |B|=r, B \subseteq V(G)} \sum_{\mu\in\mathcal{P}(n)} \chi_\lambda(\mu) \sum_{\sigma\in(\mathcal{S}_{n})_{B}\atop type(\sigma)=\mu} \prod_{i=1}^n h_{i\sigma(i)},
\end{eqnarray*}
where $C$ is a conjugate class of the subgroup $(\mathcal{S}_n)_B$.

Assume that  $B=\{v_{i_{1}}, v_{i_{2}}, \ldots, v_{i_{r}}\}$ with $1\leq i_{1}\leq i_{2}\leq \cdots\leq i_{r}\leq n$. Let $\mathcal{S}_B$ be the set of permutations on the indices of the vertices in  $B$, and let $G[B]$ be the subgraph induced on $B$.   In fact, we only consider the case that the term $\prod\limits_{i} h_{i\sigma(i)}$ is non-zero.
For any $\sigma \in (\mathcal{S}_n)_B$, we have $\sigma(i)=i$ for $v_{i}\in \{v_{1}, v_{2}, \ldots, v_{n}\}\setminus B$.  Thus, for any $\sigma \in (\mathcal{S}_n)_B$, there exists a unique permutation $\sigma'\in\mathcal{S}_B$ such that $\sigma(i) = \sigma'(i)$ for $v_{i} \in B$.   Then $\sigma \in (\mathcal{S}_n)_B$ and $\sigma' \in \mathcal{S}_B$ are in one-to-one correspondence. Therefore, it suffices to consider permutations in $\mathcal{S}_B$. The term $h_{i\sigma(i)} \neq 0$ if and only if $\sigma(i) = i$ or $\sigma(i) \neq i$ and $(i, \sigma(i)) \in E(G)$. Let $\sigma = (j_1 j_2 \cdots j_s)(j_{s+1} \cdots j_t) \cdots (j_m)(j_{m+1}) \cdots (j_r) \in \mathcal{S}_B$ be a product of disjoint cycles such that $\prod\limits_{v_{i} \in B} h_{i\sigma(i)} \neq 0$. Cycles of length $3$ or more in $\sigma$ correspond to cycles of the same length in $G[B]$. A transposition in $\sigma$ corresponds to an edge in $G[B]$.
Let $B_1 \subseteq B$ be the set of fixed vertices under $\sigma$. Let $B_2$ be the set of edges in $G[B]$ corresponding to the transpositions in the disjoint cycle factorization of $\sigma$. Let $B_3, B_4, \ldots, B_t$ be all graph cycles in $G[B]$ corresponding to cycles of length greater than $2$ in the disjoint cycle factorization of $\sigma$. Note that if $v_{i} \in B$, then $\sigma(i) \in B$ for $\sigma \in \mathcal{S}_B$.
Any $2$-cycle $(i, j)$ in $\sigma$ corresponds to the non-zero factor $h_{ij}h_{ji}$, i.e., $v_{i}v_{j} \in E(G)$. Any $l$-cycle $(j_1j_2 \cdots j_l)$ in $\sigma$ corresponds to the non-zero factor $h_{j_1j_2}h_{j_2j_3}\cdots h_{j_lj_1}$, meaning that $v_{j_1}v_{j_2}\cdots v_{j_t}v_{j_1}$ is a graph cycle of girth $r$ in $G$. Therefore, we have
\begin{eqnarray*}
\sum_{\sigma\in(\mathcal{S}_{n})_{B}\atop type(\sigma)=\mu} \prod_{v_{i} \in B} h_{i\sigma(i)}& =& \sum_{\sigma\in(\mathcal{S}_{n})_{B}\atop type(\sigma)=\mu} \left( \prod_{v_{i} \in B_1} h_{ii} \prod_{e=(v_{i},v_{\sigma(i})) \in B_2} h_{i\sigma(i)}^2 \prod_{e=(v_{i},v_{\sigma(i)}) \in B_3} h_{i,\sigma(i)} \cdots \prod_{e=(v_{i},v_{\sigma(i)}) \in B_t} h_{i,\sigma(i)} \right)\\
&=& \sum_{\sigma\in(\mathcal{S}_{n})_{B}\atop type(\sigma)=\mu} \beta^{|B_1|}\gamma^{n-|B_1|} \prod_{v_{i} \in B_1} d(v_{i})\\
&=& \sum_{\sigma\in(\mathcal{S}_{n})_{B}\atop type(\sigma)=\mu} \beta^{\mathcal{F}(\sigma)}\gamma^{n-\mathcal{F}(\sigma)} \prod_{v_{i} \in B: i=\sigma(i)} d(v_{i}).
\end{eqnarray*}

This completes the proof.
\end{proof}
From Lemma \ref{lemma3.1}, we derive formulas for the immanantal polynomial coefficients of the Laplacian matrix, the signless Laplacian matrix, and the $A_{\alpha}$ matrix of a graph. For a more concise and elegant form, the formula for the coefficients corresponding to the $A_{\alpha}$ matrix is given as follows.

\begin{corollary}
Let $G$ be an $n$-vertex graph. If $\alpha\notin \{0,1\}$, then
\begin{eqnarray*}
c_{\lambda,r}(A_{\alpha}(G))=(1-\alpha)^{n}\sum\limits_{B:|B|=r,B\subseteq V(G)}\sum\limits_{C}\chi_\lambda(C)\sum\limits_{\sigma\in C}(\frac{\alpha}{1-\alpha})^{\mathcal{F}(\sigma)}\prod\limits_{v_{i}\in B;\sigma(i)=i}d(v_{i}),
\end{eqnarray*}
where $C$ is a conjugacy class of the subgroup $(\mathcal{S}_{n})_{B}$, and $\mathcal{F}(\sigma)$ is the number of fixed points in $\sigma$.
\end{corollary}
The relationship between $\sum\limits_{\sigma\in (\mathcal{S}_{n})_{B}\atop \text{type}(\sigma)=\mu} \prod\limits_{i=1}^n h_{i\sigma(i)}$ in Lemma \ref{lemma3.1} and $a_G(\nu,r)$ is further clarified as follows.
\begin{lemma}\label{lem2.1}
Let $G$ be a graph with $n$ vertices, and let $\mu= (1^{t_1}, 2^{t_2}, 3^{t_3}, 4^{t_4}, \ldots)\in\mathcal{P}(n)$.  Then for each $r$-subset $B$ of vertices of $G$,
$$\sum\limits_{\sigma\in(\mathcal{S}_{n})_{B}\atop type(\sigma)=\mu } \prod\limits_{i=1}^n h_{i\sigma(i)}=\beta^{t_1}\gamma^{n-t_1}\sum_{\nu \in \mathcal{PB}(n)} \binom{\nu}{\mu} a_G(\nu,r),$$
where  $\nu = (1^{r_1}, 2^{r_2},3^{r_3}, 4^{r_4}, \ldots)$ with $r_i \geq t_i$ for all $i = 2,3, 4, 5, \ldots$,  and  $\binom{\nu}{\mu} = \binom{r_2}{t_2}\binom{r_3}{t_3}\binom{r_4}{t_4}\cdots$.
\end{lemma}

\begin{proof}
Given $B$, we count the set
$$
\mathscr{O}_\mu = \{(\sigma, O) : type(\sigma) =  \mu, O \text{ is a } (B)\text{-vertex orientation associated with } \sigma\}.
$$
Given a  $\sigma\in (\mathcal{S}_{n})_{B}$ with type$(\sigma) =  \mu$, we can form a $(B)$-vertex orientation $O$. The construction proceeds as follows: given a cycle $(a_1 a_2 \dots a_l)$ with $l > 2$ in a permutation $\sigma \in (\mathcal{S}_n)_B$, we orient the corresponding edges in $G[B]$ to form a directed cycle-specifically, from $v_{a_i}$ to $v_{a_{i+1}}$ for $i=1,\dots,l-1$, and finally from $v_{a_l}$ to $v_{a_1}$. When the cycle is a transposition $(a_1 a_2)$ ($l=2$), we place two arrows on the edge $\{v_{a_1}, v_{a_2}\}$, pointing in opposite directions. By this  construction, a $(B)$-vertex orientation $O$ is derived. Clearly, $(\sigma, O)\in\mathscr{O}_\mu$. This implies that
\begin{eqnarray}\label{equ3.1}
\sum\limits_{\sigma\in(\mathcal{S}_{n})_{B}\atop type(\sigma)=\mu,} \prod\limits_{i=1}^n h_{i\sigma(i)}=\beta^{t_1}\gamma^{r-t_1}|\mathfrak{M}_\mu|.
\end{eqnarray}

Given a $(B)$-vertex orientation  $O$ of type $\nu$. There are $\binom{r_2}{t_2}$ ways to pick $t_2$ independent edges, $\binom{r_3}{t_3}$ ways to pick $t_3$ $3$-cycles and so on. Each choice of edges and cycles corresponds to a permutation $\sigma\in(\mathcal{S}_{n})_{B}$ of cycle type $\mu$. So there are $\binom{r_2}{t_2}\binom{r_3}{t_3}\binom{r_4}{t_4}\cdots$ ways to form the ordered pair $(\sigma, O)$. Hence,
\begin{eqnarray}\label{equ3.2}
|\mathscr{O}_\mu|= \sum_{\nu \in \mathcal{PB}(n)} \binom{\nu}{\mu} a_G(\nu,r),
\end{eqnarray}

Combining Equations (\ref{equ3.1}) and (\ref{equ3.2}), we have
$$\sum\limits_{\sigma\in(\mathcal{S}_{n})_{B}\atop type(\sigma)=\mu} \prod\limits_{i=1}^n h_{i\sigma(i)}=\beta^{t_1}\gamma^{r-t_1}\sum_{\nu \in \mathcal{PB}(n)(n)} \binom{\nu}{\mu} a_G(\nu,r).$$
This completes the proof.
\end{proof}

This allows us to rewrite $c_{\lambda,r}(\beta D(G)+\gamma A(G))$ in terms of $a_G(\nu,r)$.
\begin{lemma}\label{prop2.1}
Let $G$ be an $n$-vertex graph. Then
\begin{eqnarray*}
c_{\lambda,r}(\beta D(G)+\gamma A(G)) &=&\sum_{\nu\in\mathcal{P}(n)}  a_G(\nu,r) \beta^{t_1}\gamma^{r-t_1}  \sum_{\mu\in\mathcal{P}(n)}  \chi_\lambda(\mu)\binom{\nu}{\mu},
\end{eqnarray*}
where $\mu = (1^{t_1}, 2^{t_2}, 3^{t_3}, 4^{t_4}, \ldots)$ and $\nu = (1^{r_1}, 2^{r_2}, 3^{r_3}, 4^{r_4}, \ldots)$ are two partitions in $\mathcal{P}(n)$ satisfying $r_i \geq t_i$ for all $i = 2,3, 4, 5, \ldots$.
\end{lemma}
\begin{proof}
From Lemma \ref{lemma3.1} and Lemma \ref{lem2.1}, we obtain
\begin{eqnarray*}
c_{\lambda,r}(\beta D(G)+\gamma A(G)) &=& \sum\limits_{B:|B|=r,B\subseteq V(G)}\sum_{\mu\in\mathcal{P}(n)} \chi_\lambda(\mu) \sum_{\sigma\in(\mathcal{S}_{n})_{B}\atop type(\sigma)=\mu} \prod_{i=1}^n h_{i\sigma(i)}\nonumber\\
&=& \sum\limits_{B:|B|=r,B\subseteq V(G)}\sum_{\mu\in\mathcal{P}(n)} \chi_\lambda(\mu) \beta^{t_1}\gamma^{r-t_1}\sum_{\nu\in\mathcal{PB}(n)} \binom{\nu}{\mu} a_G(\nu,r)\nonumber\\
&=&\sum_{\nu\in\mathcal{P}(n)}  a_G(\nu,r) \beta^{t_1}\gamma^{r-t_1}  \sum_{\mu\in\mathcal{P}(n)}  \chi_\lambda(\mu)\binom{\nu}{\mu}.
\end{eqnarray*}
We have completed the proof of Lemma \ref{prop2.1}.
\end{proof}

\begin{lemma}(Chan and Lam, \cite{chan4})\label{lem3.2}
Let $\chi_\lambda(\mu)$ be an irreducible character of $\mathcal{S}_{n}$. Given a partition $\nu$ of $n$, we have
 $$\sum_{\mu\in\mathcal{P}(n)}  \chi_\lambda(\mu)\binom{\nu}{\mu}\geq0.$$
\end{lemma}

\begin{lemma}\label{lem3.5.3}
Let $T$ be a tree with $n$ vertices. If $\beta>0$ and $\gamma\neq0$, then
\begin{eqnarray*}\label{equ3.4}
c_{\lambda,r}(\beta D(T)+\gamma A(T))\geq c_{\lambda,r}(\beta D(S_{n})+\gamma A(S_{n})).
\end{eqnarray*}
\end{lemma}
\begin{proof}
Let $\mu = (1^{t_1}, 2^{t_2}, 3^{t_3}, 4^{t_4}, \ldots)$ and $\nu = (1^{r_1}, 2^{r_2}, 3^{r_3}, 4^{r_4}, \ldots)$ be two partitions in $\mathcal{P}(n)$ satisfying $r_i \geq t_i$ for all $i = 2,3, 4, 5, \ldots$. For a tree $T$, we only consider $\mu = (1^{t_1}, 2^{t_2})$ and $\nu = (1^{r_1}, 2^{r_2})$, with $r_2 \geq t_2$. Note that $t_1+2t_2=n$. According to Lemma \ref{prop2.1}, we have
\begin{eqnarray*}
c_{\lambda,r}(\beta D(T)+\gamma A(T)) &=& \sum_{\nu\in\mathcal{P}(n)}  a_T(\nu,r) \beta^{m_1}\gamma^{2m_2}  \sum_{\mu\in\mathcal{P}(n)}  \chi_\lambda(\mu)\binom{\nu}{\mu}.
\end{eqnarray*}
From Lemma \ref{lem3.5}, $\beta>0$ and $\gamma\neq0$, we deduce
\begin{eqnarray*}
\sum_{\nu\in\mathcal{P}(n)}  a_T(\nu,r) \beta^{m_1}\gamma^{2m_2}  \sum_{\mu\in\mathcal{P}(n)}  \chi_\lambda(\mu)\binom{\nu}{\mu}&\leq&
\sum_{\nu\in\mathcal{P}(n)}  a_{S_{n}}(\nu,r) \beta^{m_1}\gamma^{2m_2}  \sum_{\mu\in\mathcal{P}(n)}  \chi_\lambda(\mu)\binom{\nu}{\mu}.\\
\end{eqnarray*}
This implies $c_{\lambda,r}(\beta D(T)+\gamma A(T))\geq c_{\lambda,r}(\beta D(S_{n})+\gamma A(S_{n}))$.
\end{proof}

\begin{lemma}\label{lem3.8}
Let $G$ be a graph, and let $G_{1}$ be a subgraph of $G$ with $V(G_{1})=V(G)$.  If $\beta>0$ and $\gamma\geq-\beta$, then
\begin{eqnarray*}
c_{\lambda,r}(\beta D(G)+\gamma A(G))\geq c_{\lambda,r}(\beta D(G_{1})+\gamma A(G_{1})).
\end{eqnarray*}
\end{lemma}
\begin{proof}
According to the value of $\gamma$, we consider two following cases.

\textbf{Case 1.} Assume that $-\beta\leq\gamma\leq0$. Given that $H(G)=\beta D(G)+\gamma A(G)$ and $H(G_{1})=\beta D(G_{1})+\gamma A(G_{1})$, consider an $r \times r$ principal submatrix $H_1$ of $H(G_1)$, whose rows and columns are indexed by a vertex subset $R \subseteq V(G_1)$. Define $G[R]$ to be the subgraph of $G$ induced by $R$, and $G_{1}[R]$ as the corresponding induced subgraph of $G_1$. Then
\begin{eqnarray*}
H_1 = H(G_{1}[R]) + D^R_1,
\end{eqnarray*}
where $D^R_1$ is a nonnegative diagonal matrix, whose entries may be zero.
Each main diagonal element of $D^R_1$ counts the number of edges joining the corresponding vertex in $G_1$ to a vertex outside $R$. Now define $G_{2}[R]$ as the graph on vertex set $R$ containing exactly those edges of $G[R]$ that are missing in $G_{1}[R]$. Then
\begin{eqnarray*}
H(G[R]) = H(G_{1}[R]) +  H(G_{2}[R]).
\end{eqnarray*}
Finally, the principal submatrix $H$ of $H(G)$ that corresponds to the vertex set $R$ is given by
\begin{eqnarray*}
H = H(G[R]) + D^R,
\end{eqnarray*}
where $D^R$ is a nonnegative diagonal matrix. Every main diagonal entry of $D^R_1$ denotes the number of edges that connect the corresponding vertex in $G_1$ to a vertex not contained in $R$. It follows that
\begin{eqnarray*}
H = H_1+H(G_{2}[R]) + (D^R-D^R_1),
\end{eqnarray*}
Since $G_1$ is a subgraph of $G$, $D^R - D^R_1$ is a diagonal matrix with nonnegative entries. By Lemma \ref{lem3.7}, we obtain that $H(G_{2}[R])$ is a positive semi-definite hermitian matrix. This indicate that $H(G_{2}[R]) + (D^R-D^R_1)$ is a positive semi-definite hermitian matrix.
Applying Lemma \ref{lem3.6}, we get
$d_{\lambda}(H)\geq d_{\lambda}(H_1)$.
This implies that we have derived  a one-to-one correspondence between the $r \times r$ principal submatrices $H_1$ of $H(G_1)$ and some of the $r \times r$ principal submatrices $H$ of $H(G)$ such that $d_{\lambda}(H)\geq d_{\lambda}(H_1)$. Thus, $c_{\lambda,r}(\beta D(G)+\gamma A(G))\geq c_{\lambda,r}(\beta D(G_{1})+\gamma A(G_{1}))$.

\textbf{Case 2.} Assume that $\gamma>0$. Let $\mu = (1^{t_1}, 2^{t_2}, 3^{t_3}, 4^{t_4}, \ldots)$ and $\nu = (1^{r_1}, 2^{r_2}, 3^{r_3}, 4^{r_4}, \ldots)$ be two partitions in $\mathcal{P}(n)$ with $r_i \geq t_i$ for all $i = 2,3, 4, 5, \ldots$. Since $\beta>0$ and $\gamma>0$, we have $\beta^{t_1}\gamma^{r-t_1}\geq0$.
By Equation (\ref{equ3.3}), Lemma \ref{lem3.2} and Lemma \ref{lem3.3}, we have
\begin{eqnarray*}
c_{\lambda,r}(\beta D(G)+\gamma A(G)) &=&\sum_{\nu\in\mathcal{P}(n)}  a_G(\nu,r) \beta^{t_1}\gamma^{k-t_1}  \sum_{\mu\in\mathcal{P}(n)}  \chi_\lambda(\mu)\binom{\nu}{\mu}\\
&\geq&\sum_{\nu\in\mathcal{P}(n)}  a_{G-e}(\nu,r) \beta^{t_1}\gamma^{k-t_1}  \sum_{\mu\in\mathcal{P}(n)}  \chi_\lambda(\mu)\binom{\nu}{\mu}\\
&=&c_{\lambda,r}(\beta D(G-e)+\gamma A(G-e)).
\end{eqnarray*}
It means that $c_{\lambda,r}(\beta D(G)+\gamma A(G))\geq c_{\lambda,r}(\beta D(G_{1})+\gamma A(G_{1}))$.

Based the above arguments, we complete the proof.
\end{proof}

\begin{theorem}\label{the3.2}
Assume that  $G$ is a graph with $n$ vertices. If $\beta>0$, $\gamma\geq-\beta$ and $\gamma\neq0$, then
$$c_{\lambda,r}(\beta D(S_{n})+\gamma A(S_{n}))\leq c_{\lambda,r}(\beta D(G)+\gamma A(G))\leq c_{\lambda,r}(\beta D(K_{n})+\gamma A(K_{n})).$$
\end{theorem}

\begin{proof}
Let $T$ be a spanning tree of $G$ of order $n$. By Lemma \ref{lem3.5.3}, we have $c_{\lambda,r}(\beta D(T)+\gamma A(T))\geq c_{\lambda,r}(\beta D(S_{n})+\gamma A(S_{n}))$. Note that $T$ is a subgraph of $G$ and $G$ is a subgraph of $K_{n}$. By Lemma \ref{lem3.8}, we obtain $c_{\lambda,r}(\beta D(S_{n})+\gamma A(S_{n}))\leq c_{\lambda,r}(\beta D(G)+\gamma A(G))\leq c_{\lambda,r}(\beta D(K_{n})+\gamma A(K_{n}))$.
\end{proof}

Using Theorem \ref{the3.2}, we can obtain the upper and lower bounds for $c_{\lambda,r}(L(G))$.

\begin{corollary}\label{cor2.1}
Let $G$ be a graph on $n$ vertices. Then
\begin{eqnarray*}
c_{\lambda,r}(L(S_{n}))\leq c_{\lambda,r}(L(G))\leq c_{\lambda,r}(L(K_{n})).
\end{eqnarray*}
\end{corollary}

Merris~\cite{mer4} proposed an open problem concerning the coefficients of the second immanantal polynomial of the Laplacian matrix of a graph: characterize the maximum and minimum values of $c_{(2,1^{n-2}),r}(L(G))$. Corollary~\ref{cor2.1} provides a solution to this problem. Moreover, Merris et al.~\cite{mer3} gave the minimum and maximum values of the permanent of the Laplacian matrix of a graph, which can be directly obtained from Corollary~\ref{cor2.1}.

\begin{corollary}(Merris et al., \cite{mer3})
Let $G$ be a graph with $n$ vertices. Then
\begin{eqnarray*}
\operatorname{per}(L(S_{n}))\leq \operatorname{per}(L(G))\leq \operatorname{per}(L(K_{n})).
\end{eqnarray*}
\end{corollary}

Applying Theorem \ref{the3.2}, we can also obtain upper and lower bounds for $c_{\lambda,r}(Q(G))$ and $c_{\lambda,r}(A_{\alpha}(G))$.

\begin{corollary}
Let $G$ be a graph on $n$ vertices. Then
\begin{eqnarray*}
c_{\lambda,r}(Q(S_{n}))\leq c_{\lambda,r}(Q(G))\leq c_{\lambda,r}(Q(K_{n})).
\end{eqnarray*}
\end{corollary}

\begin{corollary}
Let $G$ be a graph with $n$ vertices. If $\alpha\notin\{0,1\}$, then
\begin{eqnarray*}
c_{\lambda,r}(A_{\alpha}(S_{n}))\leq c_{\lambda,r}(A_{\alpha}(G))\leq c_{\lambda,r}(A_{\alpha}(K_{n})).
\end{eqnarray*}
\end{corollary}

\subsection{Bounds for the coefficients of the immanantal polynomial of the Linear combination matrix for trees and bipartite graphs}
When $G$ is a tree or a bipartite graph, we use the $(B)$-vertex orientation method to give bounds for the coefficient $c_{\lambda,r}(\beta D(G)+\gamma A(G))$.

\begin{theorem}\label{the3.1}
Let $T$ be a tree with $n$ vertices. If $\beta>0$ and $\gamma\neq0$, then
$$c_{\lambda,r}(\beta D(S_{n})+\gamma A(S_{n}))\leq c_{\lambda,r}(\beta D(T)+\gamma A(T))\leq c_{\lambda,r}(\beta D(P_{n})+\gamma A(P_{n})).$$
\end{theorem}
\begin{proof}
The left inequality follows directly from Lemma \ref{lem3.5.3}. Now we prove the right inequality. Similarly, let $\mu = (1^{t_1}, 2^{t_2}, 3^{t_3}, 4^{t_4}, \ldots)$ and $\nu = (1^{r_1}, 2^{r_2}, 3^{r_3}, 4^{r_4}, \ldots)$ be two partitions in $\mathcal{P}(n)$ satisfying $r_i \geq t_i$ for all $i = 2,3, 4, 5, \ldots$. For a tree $T$, we only consider $\mu = (1^{t_1}, 2^{t_2})$ and $\nu = (1^{r_1}, 2^{r_2})$, with $r_2 \geq t_2$. Note that $t_1+2t_2=n$. According to Theorem \ref{prop2.1}, we have
\begin{eqnarray*}
c_{\lambda,r}(\beta D(T)+\gamma A(T)) &=& \sum_{\nu\in\mathcal{P}(n)}  a_T(\nu,r) \beta^{m_1}\gamma^{2m_2}  \sum_{\mu}  \chi_\lambda(\mu)\binom{\nu}{\mu}.
\end{eqnarray*}
From Lemma \ref{lem3.5}, $\beta>0$ and $\gamma\neq0$, we deduce
\begin{eqnarray*} \sum_{\nu\in\mathcal{P}(n)}  a_T(\nu,r) \beta^{m_1}\gamma^{2m_2}  \sum_{\mu\in\mathcal{P}(n)}  \chi_\lambda(\mu)\binom{\nu}{\mu}
&\leq&\sum_{\nu\in\mathcal{P}(n)}  a_{p_{n}}(\nu,r) \beta^{m_1}\gamma^{2m_2}  \sum_{\mu\in\mathcal{P}(n)}  \chi_\lambda(\mu)\binom{\nu}{\mu}.
\end{eqnarray*}
This implies $c_{\lambda,r}(\beta D(T)+\gamma A(T))\leq c_{\lambda,r}(\beta D(P_{n})+\gamma A(P_{n}))$. The proof is complete.
\end{proof}

Chan and Lam \cite{chan5} gave upper and lower bounds for the coefficients of the immanantal polynomial of the Laplacian matrix of trees. This is a direct corollary of Theorem \ref{the3.1}.

\begin{corollary}[Chan and Lam, \cite{chan5}]
Let $T$ be a tree on $n$ vertices. Then
\begin{eqnarray*}
c_{\lambda,r}(L(S_{n}))\leq c_{\lambda,r}(L(T))\leq c_{\lambda,r}(L(P_{n})).
\end{eqnarray*}
\end{corollary}

The upper and lower bounds for $c_{\lambda,r}(A_{\alpha}(T))$ can be obtained from Theorem~\ref{the3.1}.

\begin{corollary}
Let $T$ be a tree with $n$ vertices. If $\alpha\notin\{0,1\}$, then
\begin{eqnarray*}
c_{\lambda,r}(A_{\alpha}(S_{n}))\leq c_{\lambda,r}(A_{\alpha}(T))\leq c_{\lambda,r}(A_{\alpha}(P_{n})).
\end{eqnarray*}
\end{corollary}

\begin{theorem}\label{the4.1}
Let $G$ be a bipartite graph with $n$ vertices. If $\beta>0$ and $\gamma\neq0$, then
$$c_{\lambda,k}(\beta D(S_{n})+\gamma A(S_{n}))\leq c_{\lambda,k}(\beta D(G)+\gamma A(G))\leq c_{\lambda,k}(\beta D(K_{\lceil \frac{n}{2} \rceil,\lfloor \frac{n}{2} \rfloor})+\gamma A(K_{\lceil \frac{n}{2} \rceil,\lfloor \frac{n}{2} \rfloor}))$$
\end{theorem}
\begin{proof}
Let $\mu = (1^{t_1}, 2^{t_2}, 3^{t_3}, 4^{t_4}, \ldots)$ and $\nu = (1^{r_1}, 2^{r_2}, 3^{r_3}, 4^{r_4}, \ldots)$ be two partitions in $\mathcal{P}(n)$ satisfying $r_i \geq t_i$ for all $i = 2,3, 4, 5, \ldots$.
Since $G$ is bipartite, we only consider partitions $\mu = (1^{t_1}, 2^{t_2},4^{t_4},6^{t_6}, \ldots)$ and $\nu = (1^{r_1}, 2^{r_2},4^{r_4},6^{r_6}, \ldots)$, with $n_i \geq m_i$ for all $i = 2, 4, 6, \ldots$.
Note that $t_1+2t_2+4t_4+6t_6+\cdots=n$, so it is easy to see that $n-t_1$ is even. Since $\beta>0$ and $\gamma\neq0$,
we have $\beta^{t_1}\gamma^{n-t_1}>0$. Suppose $G$ has a $(p,q)$-bipartition. From Lemma \ref{lem3.3}, we deduce
$a_{G}(\nu,r)\leq a_{K_{p,q}}(\nu,r)$. By Lemma \ref{lem2.6}, we have
$a_{K_{p,q}}(\nu,r)\leq a_{K_{\lceil \frac{n}{2} \rceil,\lfloor \frac{n}{2} \rfloor}}(\nu,r)$.
This implies
\begin{eqnarray}\label{equ4.2}
a_{G}(\nu,r)\leq a_{K_{\lceil \frac{n}{2} \rceil,\lfloor \frac{n}{2} \rfloor}}(\nu,r).
\end{eqnarray}
From equation (\ref{equ3.3}), (\ref{equ4.2}), and Lemma \ref{lem3.2}, we obtain
\begin{eqnarray}\label{equ4.3}
c_{\lambda,r}(\beta D(G)+\gamma A(G))
&=&\sum_{\nu\in\mathcal{P}(n)}  a_G(\nu,r) \beta^{m_1}\gamma^{k-m_1}  \sum_{\mu\in\mathcal{P}(n)}  \chi_\lambda(\mu)\binom{\nu}{\mu}\nonumber\\
&\leq&\sum_{\nu\in\mathcal{P}(n)}  a_{K_{\lceil \frac{n}{2} \rceil,\lfloor \frac{n}{2} \rfloor}}(\nu,r) \beta^{m_1}\gamma^{k-m_1}  \sum_{\mu\in\mathcal{P}(n)}  \chi_\lambda(\mu)\binom{\nu}{\mu}\nonumber\\
&=&c_{\lambda,r}(\beta D(K_{\lceil \frac{n}{2} \rceil,\lfloor \frac{n}{2} \rfloor})+\gamma A(K_{\lceil \frac{n}{2} \rceil,\lfloor \frac{n}{2} \rfloor})).
\end{eqnarray}
Let $T$ be a spanning tree of $G$ of order $n$. By Theorem \ref{the3.1}, we have $c_{\lambda,r}(\beta D(T)+\gamma A(T))\geq c_{\lambda,r}(\beta D(S_{n})+\gamma A(S_{n}))$. Note that $T$ is a subgraph of $G$. By Lemma \ref{lem3.3}, we obtain
\begin{eqnarray}\label{equ4.4}
c_{\lambda,r}(\beta D(G)+\gamma A(G))\geq c_{\lambda,r}(\beta D(S_{n})+\gamma A(S_{n})).
\end{eqnarray}
Combining Equations (\ref{equ4.3}) and (\ref{equ4.4}) yields the desired result. This completes the proof of Theorem \ref{the4.1}.
\end{proof}
By applying Theorem \ref{the4.1}, we obtain upper and lower bounds for $c_{\lambda,r}(L(G))$ and $c_{\lambda,r}(A_{\alpha}(G))$.

\begin{corollary}[Chan and Lam, \cite{chan4}]
Let $G$ be a bipartite graph on $n$ vertices. Then
\begin{eqnarray*}
c_{\lambda,r}(L(S_{n}))\leq c_{\lambda,r}(L(G))\leq c_{\lambda,r}(L(K_{\lceil n/2 \rceil,\lfloor n/2 \rfloor})).
\end{eqnarray*}
\end{corollary}

\begin{corollary}
Let $G$ be a bipartite graph on $n$ vertices. If $\alpha\notin\{0,1\}$, then
\begin{eqnarray*}
c_{\lambda,r}(A_{\alpha}(S_{n}))\leq c_{\lambda,r}(A_{\alpha}(G))\leq c_{\lambda,r}(A_{\alpha}(K_{\lceil n/2 \rceil,\lfloor n/2 \rfloor})).
\end{eqnarray*}
\end{corollary}

\section{The necessary and sufficient conditions for the equality of hook immanantal polynomials of linear combination matrices of regular graphs}

In this section, we first characterize the first six coefficients of the hook immanantal polynomials for the linear combination matrices of graphs. Using these coefficients, we prove that the hook immanantal polynomials of the linear combination matrices of two regular graphs are equal if and only if those of their adjacency matrices are equal. Before presenting the main results, we introduce two important lemmas.

\begin{lemma}\label{lem5.3}
Let $i$ be a positive integer with $0 < i < \lfloor n/2 \rfloor$. Then
\begin{eqnarray*}
\chi_{(k, 1^{n-k})}(2^{i},1^{n-2i}) = \sum\limits_{j=0}^{i}(-1)^{j}\binom{n-2i-1}{k-2i+2j-1}\binom{i}{j}= \sum\limits_{j=0}^{i}(-1)^{j}\binom{n-2i-1}{n-k-2j}\binom{i}{j}.
\end{eqnarray*}
\end{lemma}
\begin{proof}
By Lemma \ref{the5.1}, we have
\begin{eqnarray}\label{equ3.1}
&&\chi_{(k, 1^{n-k})}(2^{i},1^{n-2i}) \nonumber\\
&=& \chi_{(k-2, 1^{n-k})}(2^{i-1},1^{n-2i})-\chi_{(k, 1^{n-k-2})}(2^{i-1},1^{n-2i})\nonumber\\
&=& \chi_{(k-4, 1^{n-k})}(2^{i-2},1^{n-2i})-2\chi_{(k-2, 1^{n-k-2})}((2^{i-2},1^{n-2i}))+\chi_{(k, 1^{n-k-4})}(2^{i-2},1^{n-2i})\nonumber\\
&=& \chi_{(k-6, 1^{n-k})}(2^{i-3},1^{n-2i})-3\chi_{(k-4, 1^{n-k-2})}(2^{i-3},1^{n-2i})\nonumber\\
&&+3\chi_{(k-2, 1^{n-k-4})}(2^{i-3},1^{n-2i})+\chi_{(k, 1^{n-k-6})}(2^{i-3},1^{n-2i})\nonumber\\
&=& \chi_{(k-8, 1^{n-k})}(2^{i-4},1^{n-2i})-4\chi_{(k-6, 1^{n-k-2})}(2^{i-4},1^{n-2i})+6\chi_{(k-4, 1^{n-k-4})}(2^{i-4},1^{n-2i})\nonumber\\
&&-4\chi_{(k-2, 1^{n-k-6})}(2^{i-4},1^{n-2i})+\chi_{(k, 1^{n-k-8})}(2^{i-4},1^{n-2i})\nonumber\\
&=& \cdots\nonumber\\
&=& \binom{i}{0}\cdot \chi_{(k-2i, 1^{n-k})}(1^{n-2i})-\binom{i}{1}\cdot\chi_{(k-2i+2, 1^{n-k-2})}(1^{n-2i})\nonumber\\
&&+\binom{i}{2}\cdot\chi_{(k-2i+4, 1^{n-k-4})}(1^{n-2i})+\cdots+(-1)^{j}\binom{i}{j}\cdot\chi_{(k-2i+2j, 1^{n-k-2j})}(1^{n-2i})\nonumber\\
&&+\cdots+(-1)^{i-1}\binom{i}{i-1}\cdot\chi_{(k-2, 1^{n-k-2i+2})}((1^{n-2i}))+(-1)^{i}\binom{i}{i}\cdot\chi_{(k, 1^{n-k-2i})}(1^{n-2i})\nonumber\\
&=& \sum\limits_{j=0}^{i}(-1)^{j}\binom{i}{j}\cdot\chi_{(k-2i+2j, 1^{n-k-2j})}(1^{n-2i}).
\end{eqnarray}
By Lemma \ref{lem5.1}, we obtain
\begin{eqnarray}\label{equ3.2}
\chi_{(k-2i+2j, 1^{n-k-2j})}(1^{n-2i})=\binom{n-2i-1}{k-2i+2j-1}= \binom{n-2i-1}{n-k-2j}.
\end{eqnarray}
Combining Equations (\ref{equ3.1}) and (\ref{equ3.2}), we have
\begin{eqnarray*}
\chi_{(k, 1^{n-k})}(2^{i},1^{n-2i}) =\sum\limits_{j=0}^{i}(-1)^{j}\binom{n-2i-1}{n-k-2j}\binom{i}{j}.
\end{eqnarray*}
This completes the proof.
\end{proof}

 \begin{lemma}\label{lem5.4}
Let $\sigma$ be a permutation in the symmetric group $\mathcal{S}_{n}$ of type $(\sigma)=(3,2,1^{n-5})$. Then
\begin{eqnarray*}
\chi_{(k, 1^{n-k})}(\sigma)=\binom{n-6}{k-6}-\binom{n-6}{k-4}+\binom{n-6}{k-3}-\binom{n-6}{k-1}.
\end{eqnarray*}
\end{lemma}
\begin{proof}
By Lemma \ref{the5.1} and Lemma \ref{lem5.1}, we obtain
\begin{eqnarray*}
&&\chi_{(k, 1^{n-k})}(3,2,1^{n-5})\\
&=&\chi_{(k-3, 1^{n-k})}(2,1^{n-5})+\chi_{(k, 1^{n-k-3})}(2,1^{n-5})\\
&=&\chi_{(k-5, 1^{n-k})}(1^{n-5})-\chi_{(k-3, 1^{n-2-k})}(1^{n-5})+\chi_{(k-2, 1^{n-k-3})}(1^{n-5})-\chi_{(k, 1^{n-k-5})}(1^{n-5})\\
&=&\binom{n-6}{k-6}-\binom{n-6}{k-4}+\binom{n-6}{k-3}-\binom{n-6}{k-1}.
\end{eqnarray*}
We have completed the proof.
\end{proof}

 \begin{theorem}\label{theorem4.1}
Let $G$ be a graph with $n$ vertices and $m$ edges, and $(d(v_{1}),d(v_{2}),\ldots,d(v_{n}))$ denote the degree sequence of $G$. Denote
\begin{eqnarray*}
\operatorname{Imm}_{(k,1^{n-k})}(xI - \beta D(G) - \gamma A(G)) = \sum_{r=0}^{n} (-1)^r c_{(k,1^{n-k}),r}(\beta D(G) + \gamma A(G))x^{n-r}.
\end{eqnarray*}
Then
\begin{eqnarray*}
c_{(k,1^{n-k}),0}(\beta D(G) + \gamma A(G)) &=& \binom{n-1}{k-1}, \\
c_{(k,1^{n-k}),1}(\beta D(G) + \gamma A(G)) &=& F_1(G)\binom{n-1}{k-1} = 2m\beta\binom{n-1}{k-1}, \\
c_{(k,1^{n-k}),2}(\beta D(G) + \gamma A(G)) &=& \beta^{2}F_2(G)\binom{n-1}{k-1} +\bigg[\binom{n-2}{k-2}-\binom{n-2}{k-1}\bigg]m\gamma^{2}\quad (n \geq 3), \\
c_{(k,1^{n-k}),3}(\beta D(G) + \gamma A(G)) &=& \beta^{3}F_3(G)\binom{n-1}{k-1} +\beta\gamma^{2}\mathcal{M}_3^{1}(G)\bigg[\binom{n-2}{k-2}-\binom{n-2}{k-1}\bigg]\\
&&+2|\mathscr{C}_{3}(G)|\gamma^{3}\bigg[\binom{n-4}{k-4}-\binom{n-4}{k-1}\bigg]  \quad (n \geq 4),\\
c_{(k,1^{n-k}),4}(\beta D(G) + \gamma A(G)) &=& \beta^{4}F_4(G)\binom{n-1}{k-1} +\beta^{2}\gamma^{2}\mathcal{M}_4^{1}(G)\bigg[\binom{n-2}{k-2}-\binom{n-2}{k-1}\bigg]\\
&&+2\beta\gamma^{3}\mathcal{C}_{4}^{3}(G)\bigg[\binom{n-4}{k-4}-\binom{n-4}{k-1}\bigg]+2|\mathscr{C}_{4}(G)|\gamma^{4}\bigg[\binom{n-5}{k-5}-\binom{n-5}{k-1}\bigg] \\
&&+\bigg[\binom{m}{2}-\sum\limits_{i=1}^{n} \binom{d(v_{i})}{2}\bigg]\bigg[\binom{n-5}{k-5}-2\binom{n-5}{k-3}+\binom{n-5}{k-1}\bigg]\gamma^{4} \quad (n \geq 5),
\end{eqnarray*}
and
\begin{eqnarray*}
&&c_{(k,1^{n-k}),5}(\beta D(G) + \gamma A(G))\\
 &=& \beta^{5}F_{5}(G)\binom{n-1}{k-1} +\beta^{3}\gamma^{2}\mathcal{M}_5^{1}(G)\bigg[\binom{n-2}{k-2}-\binom{n-2}{k-1}\bigg]\\
&&+2\beta^{2}\gamma^{3}\mathcal{C}_{5}^{3}(G)\bigg[\binom{n-4}{k-4}+\binom{n-4}{k-1}\bigg]+2\beta\gamma^{4}\mathcal{C}_{5}^{4}(G)\bigg[\binom{n-5}{k-5}-\binom{n-5}{k-1}\bigg] \\
&&+\beta\gamma^{4}\mathcal{M}_{5}^{2}(G)\bigg[\binom{n-5}{k-5}-2\binom{n-5}{k-3}+\binom{n-5}{k-1}\bigg]+ 2|\mathscr{C}_{5}(G)|\gamma^{5}\bigg[\binom{n-6}{k-6}-\binom{n-6}{k-1}\bigg]\\
&&+\sum\limits_{j=1}^{|\mathscr{C}_{3}(G)|} 2(m+3-\mathscr{T}_{j}(G))\gamma^{5}\bigg[\binom{n-6}{k-6}+\binom{n-6}{k-3}-\binom{n-6}{k-4}-\binom{n-6}{k-1}\bigg] \quad (n \geq 6).
\end{eqnarray*}
\end{theorem}

\begin{proof}
From Equation (\ref{equ1.1}) with $M = xI - \beta D(G) - \gamma A(G)$, we get $c_{(k,1^{n-k}),0}(\beta D(G) + \gamma A(G)) = \chi_{(k, 1^{n-k})}(\text{id})$. By Lemma \ref{lem5.1}, we obtain $c_{(k,1^{n-k}),0}(\beta D(G) + \gamma A(G)) = \binom{n-1}{k-1}$.

The coefficient $c_{(k,1^{n-k}),1}(\beta D(G) + \gamma A(G))$ is $\chi_{(k, 1^{n-k})}(\text{id})$ multiplied by the trace of matrix $\beta D(G) + \gamma A(G)$. By Lemma \ref{lem5.1}, we have $c_{(k,1^{n-k}),1}(\beta D(G) + \gamma A(G))= F_1(G)\chi_{(k, 1^{n-k})}(\text{id})= F_1(G)\binom{n-1}{k-1} = 2m\beta\binom{n-1}{k-1}$.

If $n \geq 3$, contributions to $x^{n-2}$ come from two sources. The first is $\sigma = \text{id}$, accounting for the term $\beta^{2}F_2(G)\binom{n-1}{k-1}$. The second source is the set of transpositions that swap vertices forming an edge. In this case, by Lemma \ref{lem5.3}, the value of $\chi_{(k, 1^{n-k})}$ is $\binom{n-2}{k-2}-\binom{n-2}{k-1}$. So, $c_{(k,1^{n-k}),2}(\beta D(G) + \gamma A(G)) = \beta^{2}F_2(G)\binom{n-1}{k-1}+ m\gamma^{2}\big[\binom{n-2}{k-2}-\binom{n-2}{k-1}\big]$.

To complete the proof, it suffices to establish the following claim.
\begin{claim}\label{cla5.1}
Let $l$ is a positive integer with $l<n$. The contribution of an $l$-cycle to $c_{(k,1^{n-k}),r}(\beta D(G) + \gamma A(G))$ is
\begin{eqnarray*}
2|\mathscr{C}_{l}(G)|\gamma^{l}\bigg[\binom{n-l-1}{k-l-1}+(-1)^{l-1}\binom{n-l-1}{k-1}\bigg].
\end{eqnarray*}
\end{claim}
\textbf{Proof of Claim \ref{cla5.1}.}
This term comes from an $l$-cycle $\sigma = (i_{1}i_{2}\cdots i_{l})$, for which $v_{i_{1}}v_{i_{2}}$, $v_{i_{2}}v_{i_{3}}$, \ldots, $v_{i_{l-1}}v_{i_{l}}$ and $v_{i_{l}}v_{i_{1}}$ are all edges of $G$. By Lemma \ref{lem5.2}, the value of $\chi_{(k, 1^{n-k})}$ on the $l$-cycle $\sigma$ is given by $\binom{n-l-1}{k-l-1}+(-1)^{l-1}\binom{n-l-1}{k-1}$.
Note that the inverse of an $l$-cycle is also an $l$-cycle, so we deduce that each $l$-cycle in $G$ contributes both clockwise and counterclockwise orientations.
Moreover, since the entries of $\beta D(G) + \gamma A(G)$ corresponding to the edges of this $l$-cycle at positions $(i_{1},i_{2})$, $(i_{2},i_{3})$, \ldots, $(i_{l-1},i_{l})$ and $(i_{l},i_{1})$ are all $\gamma$. In any case, the total contribution of $l$-cycles to $c_{(k,1^{n-k}),r}(\beta D(G) + \gamma A(G))$ is $2|\mathscr{C}_{l}(G)|\gamma^{l}\big[\binom{n-l-1}{k-l-1}+(-1)^{l-1}\binom{n-l-1}{k-1}\big]$.

Contributions to $c_{(k,1^{n-k}),3}(\beta D(G) + \gamma A(G))$ come from three sources: $\beta^{3}F_3(G)\binom{n-1}{k-1}$ from $\sigma = \text{id}$, $\beta\gamma^{2}\mathcal{M}_3^{1}(G)\big[\binom{n-2}{k-2}-\binom{n-2}{k-1}\big]$ from transpositions swapping ``endpoints" of edges, and the remaining term from $3$-cycles. By Claim \ref{cla5.1}, we deduce the contribution of $3$-cycles is $2|\mathscr{C}_{3}(G)|\gamma^{3}\big[\binom{n-4}{k-4}+\binom{n-4}{k-1}\big]$. Therefore, it is easy to obtain $c_{(k,1^{n-k}),3}(\beta D(G) + \gamma A(G)) = \beta^{3}F_3(G)\binom{n-1}{k-1} +\beta\gamma^{2}\mathcal{M}_3^{1}(G)\big[\binom{n-2}{k-2}-\binom{n-2}{k-1}\big]+2|\mathscr{C}_{3}(G)|\gamma^{3}\big[\binom{n-4}{k-4}-\binom{n-4}{k-1}\big]$.

If $n \geq 5$, contributions to $x^{n-4}$ come from five sources. The first is $\sigma = \text{id}$, accounting for the term $\beta^{4}F_4(G)\binom{n-1}{k-1}$. The second source is the set of transpositions swapping vertices forming an edge. In this case, by Lemma \ref{lem5.3}, the value of this contribution is $\beta^{2}\gamma^{2}\mathcal{M}_4^{1}\big[\binom{n-2}{k-2}-\binom{n-2}{k-1}\big]$.
The third source is the set of transpositions swapping vertices forming a $3$-cycle. By Lemma \ref{lem5.2}, we obtain this contribution as $\beta\gamma^{3}\mathcal{C}_{4}^{3}(G)\big[\binom{n-4}{k-4}-\binom{n-4}{k-1}\big]$.
The fourth source is $4$-cycles. By Claim \ref{cla5.1}, this contribution is $2|\mathscr{C}_{4}(G)|\gamma^{4}\big[\binom{n-5}{k-5}-\binom{n-5}{k-1}\big]$. The last source is $2$-matchings. By Lemma \ref{lem5.2}, we have $\big[\binom{m}{2}-\sum\limits_{i=1}^{n} \binom{d(v_{i})}{2}\big]\big[\binom{n-5}{k-5}-2\binom{n-5}{k-3}+\binom{n-5}{k-1}\big]\gamma^{4}$ as the contribution from $2$-matchings. In total, we have $c_{(k,1^{n-k}),4}(\beta D(G) + \gamma A(G)) = \beta^{4}F_4(G)\binom{n-1}{k-1} +\beta^{2}\gamma^{2}\mathcal{M}_4^{1}\big[\binom{n-2}{k-2}-\binom{n-2}{k-1}\big]+\beta\gamma^{3}\mathcal{C}_{4}^{3}(G)\big[\binom{n-4}{k-4}-\binom{n-4}{k-1}\big]+2|\mathscr{C}_{4}(G)|\gamma^{4}\big[\binom{n-5}{k-5}-\binom{n-5}{k-1}\big] +\big[\binom{m}{2}-\sum\limits_{i=1}^{n} \binom{d(v_{i})}{2}\big]\big[\binom{n-5}{k-5}-2\binom{n-5}{k-3}+\binom{n-5}{k-1}\big]\gamma^{4}$.

Contributions to $c_{(k,1^{n-k}),5}(\beta D(G) + \gamma A(G))$ come from the following sources. $\beta^{5}F_{5}(G)\binom{n-1}{k-1}$ from $\sigma=\text{id}$. $\beta^{3}\gamma^{2}\mathcal{M}_5^{1}(G)\big[\binom{n-2}{k-2}-\binom{n-2}{k-1}\big]$ from transpositions swapping ``endpoints" of edges. By Claim \ref{cla5.1}, we deduce the contributions of $3$-cycles and $4$-cycles are $2\beta^{2}\gamma^{3}\mathcal{C}_{5}^{3}(G)\big[\binom{n-4}{k-4}+\binom{n-4}{k-1}\big]$ and $2\beta\gamma^{4}\mathcal{C}_{5}^{4}(G)\big[\binom{n-5}{k-5}-\binom{n-5}{k-1}\big]$, respectively.
The term $\beta\gamma^{4}\mathcal{M}_{5}^{2}(G)\big[\binom{n-5}{k-5}-2\binom{n-5}{k-3}+\binom{n-5}{k-1}\big]$ represents the contribution from $2$-matchings, as given by Lemma \ref{lem5.2}. Claim \ref{cla5.1} directly implies that the contribution value of $5$-cycles is given by $2|\mathscr{C}_{5}(G)|\gamma^{5}\big[\binom{n-6}{k-6}-\binom{n-6}{k-1}\big]$.
Moreover, by Lemma \ref{lem5.4}, the total contribution from all disjoint unions of one $3$-cycle and one $P_{2}$ is expressed as
\begin{eqnarray*}
\sum\limits_{j=1}^{|\mathscr{C}_{3}(G)|} (m+3-\mathscr{T}_{j}(G))\gamma^{5}\bigg[\binom{n-6}{k-6}+\binom{n-6}{k-3}-\binom{n-6}{k-4}-\binom{n-6}{k-1}\bigg]
\end{eqnarray*}
Therefore, we have $c_{(k,1^{n-k}),5}(\beta D(G) + \gamma A(G))= \beta^{5}F_{5}(G)\binom{n-1}{k-1} +\frac{(2k-n-1)\beta^{3}\gamma^{2}\mathcal{M}_5^{1}(G)}{k-1}\binom{n-2}{k-2}+2\beta^{2}\gamma^{3}\mathcal{C}_{5}^{3}(G)\\
\big[\binom{n-4}{k-4}+\binom{n-4}{k-1}\big]+2\beta\gamma^{4}\mathcal{C}_{5}^{4}(G)\big[\binom{n-5}{k-5}-\binom{n-5}{k-1}\big] +\beta\gamma^{4}\mathcal{M}_{5}^{2}(G)\big[\binom{n-5}{k-5}-2\binom{n-5}{k-3}+\binom{n-5}{k-1}\big]+ 2|\mathscr{C}_{5}(G)|\gamma^{5}\big[\binom{n-6}{k-6}-\binom{n-6}{k-1}\big]+\sum\limits_{j=1}^{|\mathscr{C}_{3}(G)|} (m+3-\mathscr{T}_{j}(G))\gamma^{5}\big[\binom{n-6}{k-6}+\binom{n-6}{k-3}-\binom{n-6}{k-4}-\binom{n-6}{k-1}\big]$.

Based on the above arguments, we complete the proof of Theorem \ref{theorem4.1}.
\end{proof}

The equivalence between the hook immanantal polynomials of $A(G)$ and $\beta D(G) + \gamma A(G)$ is characterized as follows.
\begin{theorem}\label{theorem4.2}
Let $\beta$ and $\gamma$ be two nonzero real numbers. Let $G$ and $G'$ be two regular graphs. Then the following two statements are equivalent:
\begin{enumerate}
\renewcommand{\labelenumi}{(\roman{enumi})}
\item $\operatorname{Imm}_{(k,1^{n-k})}(xI -A(G))= \operatorname{Imm}_{(k,1^{n-k})}(xI -A(G'))$.
\item $\operatorname{Imm}_{(k,1^{n-k})}(xI -\beta D(G)-\gamma A(G))= \operatorname{Imm}_{(k,1^{n-k})}(xI -\beta D(G')-\gamma A(G'))$.
\end{enumerate}
\end{theorem}
\begin{proof}
Assume $(i)$ holds. This means the polynomials $\operatorname{Imm}_{(k,1^{n-k})}(xI -A(G))$ and $\operatorname{Imm}_{(k,1^{n-k})}(xI -A(G'))$ have the same coefficients for $x^{n}$ and $x^{n-1}$. By Theorem \ref{theorem4.1}, we obtain $|V(G_1)|=|V(G')|$ and $|E(G_1)|=|E(G')|$. This implies $G$ and $G'$ are both $\frac{2|E(G)|}{|V(G)|}$-regular graphs (or $\frac{2|E(G')|}{|V(G')|}$-regular graphs). Since $\operatorname{Imm}_{(k,1^{n-k})}(xI -A(G))= \operatorname{Imm}_{(k,1^{n-k})}(xI -A(G'))$, we have $\operatorname{Imm}_{(k,1^{n-k})}(xI -\gamma A(G))= \operatorname{Imm}_{(k,1^{n-k})}(xI -\gamma A(G'))$. Replacing $x$ by $(x-\beta\frac{2|E(G)|}{|V(G)|})$ in $\operatorname{Imm}_{(k,1^{n-k})}(xI -\gamma A(G))= \operatorname{Imm}_{(k,1^{n-k})}(xI -\gamma A(G'))$ yields $\operatorname{Imm}_{(k,1^{n-k})}(xI -\beta D(G)-\gamma A(G))= \operatorname{Imm}_{(k,1^{n-k})}(xI -\beta D(G')-\gamma A(G'))$, showing $(ii)$ holds. The reverse steps show $(ii)$ implies $(i)$.
\end{proof}

Merris \cite{mer2} proved that the immanantal polynomials of the Laplacian matrices  of two regular graphs are equal if and only if the immanantal polynomials of their adjacency matrices are equal, which follows directly from Theorem \ref{theorem4.2}.

\begin{corollary}(Merris, \cite{mer2})
Let $G$ and $G'$ be two regular graphs. Then the following two statements are equivalent:
\begin{enumerate}
\renewcommand{\labelenumi}{(\roman{enumi})}
\item $\operatorname{Imm}_{(2,1^{n-2})}(xI -A(G))= \operatorname{Imm}_{(2,1^{n-2})}(xI -A(G'))$.
\item $\operatorname{Imm}_{(2,1^{n-2})}(xI -L(G))= \operatorname{Imm}_{(2,1^{n-2})}(xI -L(G'))$.
\end{enumerate}
\end{corollary}

When $\beta=1$ and $\gamma=1$, the following corollary can be obtained.
\begin{corollary}
Let $G$ and $G'$ be two regular graphs. Then the following two statements are equivalent:
\begin{enumerate}
\renewcommand{\labelenumi}{(\roman{enumi})}
\item $\operatorname{Imm}_{(2,1^{n-2})}(xI -A(G))= \operatorname{Imm}_{(2,1^{n-2})}(xI -A(G'))$.
\item $\operatorname{Imm}_{(2,1^{n-2})}(xI -Q(G))= \operatorname{Imm}_{(2,1^{n-2})}(xI -Q(G'))$.
\end{enumerate}
\end{corollary}

When $0<\beta<1$ and $\gamma=1-\beta$, a corollary concerning the $A_{\alpha}$ matrix can be derived.

\begin{corollary}
Let $G$ and $G'$ be two regular graphs. If $\alpha\notin\{0,1\}$, then the following two statements are equivalent:
\begin{enumerate}
\renewcommand{\labelenumi}{(\roman{enumi})}
\item $\operatorname{Imm}_{(2,1^{n-2})}(xI -A(G))= \operatorname{Imm}_{(2,1^{n-2})}(xI -A(G'))$.
\item $\operatorname{Imm}_{(2,1^{n-2})}(xI -A_{\alpha}(G))= \operatorname{Imm}_{(2,1^{n-2})}(xI -A_{\alpha}(G'))$.
\end{enumerate}
\end{corollary}

\begin{corollary}(Merris, \cite{mer2})
Let $G$ and $G'$ be two regular graphs. Then the following two statements are equivalent:
\begin{enumerate}
\renewcommand{\labelenumi}{(\roman{enumi})}
\item $\operatorname{Imm}_{(2,1^{n-2})}(xI -A(G))= \operatorname{Imm}_{(2,1^{n-2})}(xI -A(G'))$.
\item $\operatorname{Imm}_{(2,1^{n-2})}(xI -L(G))= \operatorname{Imm}_{(2,1^{n-2})}(xI -L(G'))$.
\end{enumerate}
\end{corollary}

\section{The lower bound for the multiplicity of the root $\beta$ of  immanantal polynomials of  linear combination  matrices of graphs}

In this section, we investigates the roots of the immanantal polynomials associated with  linear combination matrices of graphs. First, a generalization of the Frobenius-K\"{o}nig theorem and the Laplace expansion theorem to immanants is given. Utilizing these two theorems, we establish that the star degree of a graph serves as a lower bound for the multiplicity of the root $\beta$ in the immanantal polynomial of its linear combination matrix.

\begin{theorem}\label{the2.3}
Let $M$ be a square matrix of order $n$. If $M$ contains a $y\times z$ zero submatrix with $y+z=n+1$, then
$\operatorname{Imm}_{\lambda}(M) = 0$
\end{theorem}
\begin{proof}
In Equation (\ref{equ1.1}), we consider a term $\chi_\lambda(\sigma)m_{1\sigma(1)}m_{2\sigma(2)}\cdots m_{n\sigma(n)}$ in the expansion of $\operatorname{Imm}_{\lambda}(M)$, where $m_{ij}$ is the $(i,j)$ entry of $M$.
 Assume this zero submatrix consists of $y$ rows and $z$ columns, denoted by sets $\mathcal{Y}$ and $\mathcal{Z}$, respectively. This means for any $i \in \mathcal{Y}$ and $j \in \mathcal{Z}$, we have $m_{ij} = 0$.
Since $y + z = n + 1$, we have $y = n - z + 1$.
Consider any permutation $\sigma$, which maps the $y$ row indices in $\mathcal{Y}$ to $n$ column indices. The number of nonzero columns (i.e., columns not in $\mathcal{Z}$) is $n-z$. By the pigeonhole principle, at least one of the column indices obtained by mapping the elements of $\mathcal{Y}$ via $\sigma$ must belong to $\mathcal{Z}$. This means there exists $i_0 \in \mathcal{Y}$ such that $\sigma(i_0) \in \mathcal{Z}$, and then the corresponding element $m_{i_0\sigma(i_0)} = 0 $. Clearly, the term corresponding to this permutation is zero.
Note that the above conclusion holds for all permutations, so we infer that all terms in the expansion of $\operatorname{Imm}_{\lambda}(M)$ are zero, hence $\operatorname{Imm}_{\lambda}(M) = 0$.
\end{proof}

\begin{definition}\label{def3.1}
Let $\mathcal{R}$ and $\mathcal{U}$ be two subsets of $\{1,2,\dots,n\}$, with complements $\mathcal{R}' = \{1,\dots,n\} \setminus \mathcal{R}$ and $\mathcal{U}' = \{1,\dots,n\} \setminus \mathcal{U}$. The submatrix $M[\mathcal{R},\mathcal{U}]$ (respectively $M[\mathcal{R}',\mathcal{U}']$) is formed by selecting from matrix $M$ rows with indices in set $\mathcal{R}$ (respectively $\mathcal{R}'$) and columns with indices in set $\mathcal{U}$ (respectively $\mathcal{U}'$).
\end{definition}

\begin{theorem}\label{the5}
Let $M=[m_{ij}]$ be an $n \times n$ matrix, and $\mu \vdash r$, $\nu \vdash n-r$. Suppose $\mathcal{R}$, $\mathcal{R}'$, $\mathcal{U}$, $\mathcal{U}'$, $M[\mathcal{R}, \mathcal{U}]$ and $M[\mathcal{R}',\mathcal{U}']$ are as defined in Definition \ref{def3.1}. Then
\begin{eqnarray*}
\operatorname{Imm}_{\lambda}(M) &=&\sum_{\substack{\mathcal{U} \subseteq \{1,\dots,n\} \\ |\mathcal{U}| = r}}
\sum_{\substack{\mu \vdash r \\ \nu \vdash n-r}}
c_{\mu, \nu}^\lambda  d_{\mu}(M[\mathcal{R},\mathcal{U}])  d_{\nu}(M[\mathcal{R}',\mathcal{U}']),
\end{eqnarray*}
where $c_{\mu, \nu}^\lambda$ are Littlewood--Richardson coefficients.
\end{theorem}
\begin{proof}
Fix an $r$-element row subset $\mathcal{R} \subseteq \{1,2,\dots,n\}$, whose complement $\mathcal{R}' = \{1,\dots,n\} \setminus \mathcal{R}$ contains $n-r$ elements.
Any permutation
$\sigma \in \mathcal{S}_n$
can be decomposed into two independent bijections:
\begin{eqnarray*}
&&\sigma_1 = \sigma|_\mathcal{R} : \mathcal{R} \to \mathcal{U}, \quad \text{where } \mathcal{U} \text{ is an } r\text{-element column subset,}\\
&&\sigma_2 = \sigma|_{\mathcal{U}'} : \mathcal{R}' \to \mathcal{U}', \quad\text{where}~  \mathcal{U}' = \{1,\dots,n\} \setminus \mathcal{R}  \text{ contains} ~n-r~ \text{elements}.
\end{eqnarray*}
Let $S(\mathcal{R},\mathcal{U})$ and $S(\mathcal{R}',\mathcal{U}')$ denote the sets consisting of all such $\sigma_1$ and $\sigma_2$, respectively.
This implies
\begin{eqnarray}\label{equ3.1}
\sigma = \sigma_1 \sigma_2,
\end{eqnarray}
and all permutations can be partitioned according to the column subset $\mathcal{U}$ into disjoint unions:
\begin{eqnarray}\label{equ3.2}
\mathcal{S}_n = \bigsqcup_{\substack{\mathcal{U} \subseteq \{1,\dots,n\} \\ |\mathcal{U}| = r}} \mathcal{S}(\mathcal{R},\mathcal{U}) \times \mathcal{S}(\mathcal{R}',\mathcal{U}').
\end{eqnarray}
From Equations (\ref{equ3.1}) and (\ref{equ3.2}), we have
\begin{eqnarray}\label{equ3.3}
\operatorname{Imm}_{\lambda}(M) &=& \sum_{\sigma \in \mathcal{S}_n} \chi_\lambda(\sigma) \prod_{i=1}^n m_{i\sigma(i)}\nonumber\\
&=& \sum_{\substack{\mathcal{U} \subseteq \{1,\dots,n\} \\ |\mathcal{U}| = r}} \sum_{\sigma_1 \in \mathcal{S}(\mathcal{R},\mathcal{U})} \sum_{\sigma_2 \in \mathcal{S}(\mathcal{R}',\mathcal{U}')}
\chi(\sigma_1 \sigma_2)  \prod_{i \in \mathcal{R}} m_{i\sigma_1(i)}  \prod_{i \in \mathcal{R}'} m_{i\sigma_2(i)}.
\end{eqnarray}
According to the definition of immanant, we derive
\begin{eqnarray*}
\operatorname{Imm}_{\mu}(M[\mathcal{R},\mathcal{U}]) = \sum_{\sigma_1 \in \mathcal{S}(\mathcal{R},\mathcal{U})} \chi_{\mu}(\sigma_1) \prod_{i \in \mathcal{R}} m_{i\sigma_1(i)}
\end{eqnarray*}
and
\begin{eqnarray*}
\operatorname{Imm}_{\nu}(M[\mathcal{R}',\mathcal{U}']) = \sum_{\sigma_2 \in \mathcal{S}(\mathcal{R}',\mathcal{U}')} \chi_{\nu}(\sigma_2) \prod_{i \in \mathcal{R}'} m_{i\sigma_2(i)}.
\end{eqnarray*}
From Lemma \ref{lemma2.13} and Equation (\ref{equ3.3}), we obtain
\begin{eqnarray*}
\operatorname{Imm}_\lambda(M) &=& \sum_{\substack{\mathcal{U} \subseteq \{1,\dots,n\} \\ |\mathcal{U}| = r}} \sum_{\sigma_1 \in \mathcal{S}(\mathcal{R},\mathcal{U})} \sum_{\sigma_2 \in \mathcal{S}(\mathcal{R}',\mathcal{U}')}
\chi(\sigma_1  \sigma_2)  \prod_{i \in \mathcal{R}} m_{i\sigma_1(i)}  \prod_{i \in \mathcal{R}'} m_{i\sigma_2(i)}\\
&=&\sum_{\substack{\mathcal{U} \subseteq \{1,\dots,n\} \\ |\mathcal{U}| = r}} \sum_{\sigma_1 \in \mathcal{S}(\mathcal{R},\mathcal{U})} \sum_{\sigma_2 \in \mathcal{S}(\mathcal{R}',\mathcal{U}')}
\sum_{\substack{\mu \vdash r \\ \nu \vdash n-r}}
c_{\mu, \nu}^\lambda
\; \chi_{\mu}(\sigma_1) \, \chi_{\nu}(\sigma_2)  \prod_{i \in \mathcal{R}} m_{i\sigma_1(i)}  \prod_{i \in \mathcal{R}'} m_{i\sigma_2(i)}\\
&=&\sum_{\substack{\mathcal{U} \subseteq \{1,\dots,n\} \\ |\mathcal{U}| = r}}
\sum_{\substack{\mu \vdash r \\ \nu \vdash n-r}}
c_{\mu, \nu}^\lambda  \Bigg\{
\sum_{\mu \in \mathcal{S}(\mathcal{R},\mathcal{U})} \chi_{\mu}(\sigma_1) \prod_{i \in \mathcal{R}} m_{i\sigma_1(i)} \Bigg\} \Bigg\{
\sum_{\nu \in \mathcal{S}(\mathcal{R}',\mathcal{U}')} \chi_{\nu}(\sigma_2) \prod_{i \in \mathcal{R}'} m_{i\sigma_2(i)}\Bigg\}\\
&=&\sum_{\substack{\mathcal{U} \subseteq \{1,\dots,n\} \\ |\mathcal{U}| = r}}
\sum_{\substack{\mu \vdash r \\ \nu \vdash n-r}}
c_{\mu, \nu}^\lambda  \operatorname{Imm}_{\mu}(M[\mathcal{R},\mathcal{U}])  \operatorname{Imm}_{\nu}(M[\mathcal{R}',\mathcal{U}']).
\end{eqnarray*}
The proof is complete.
\end{proof}
We state the special case of the theorem for $r=1$ as follows, to be used in later proofs.
\begin{corollary}\label{cor3.1}
Let $M=[m_{ij}]$ be an $n \times n$ matrix. Then
\begin{eqnarray*}
\operatorname{Imm}_\lambda(M) = \sum\limits_{j=1}^n c_{(1), \nu}^\lambda m_{ij}  \operatorname{Imm}_{\nu}(M[\{1,\ldots,n\}\setminus\{i\},\{1,\ldots,n\}\setminus\{j\}]),
\end{eqnarray*}
where $\nu$ is the partition obtained by removing a rim hook of size $1$ from the Young diagram corresponding to $\lambda$, and $c_{(1), \nu}^\lambda$ are Littlewood--Richardson coefficients.
\end{corollary}

From Equality (\ref{equ1.2}) and Lemma \ref{pro3.2}, we can deduce the following conclusion. Suppose $M$ is an $n \times n$ matrix, and $\operatorname{Imm}_{\lambda}(xI - M)$ is defined as in Equality (\ref{equ1.2}). If $c_{\lambda,n}(M) = c_{\lambda,n-1}(M) = \cdots = c_{\lambda,n-(p-1)}(M) = 0$, then $0$ is a root of $\operatorname{Imm}_{\lambda}(xI - M)$ with multiplicity at least $p$. Suppose $M = \beta D(G) + \gamma A(G) - \beta I$, and $0$ is a root of $\operatorname{Imm}_{\lambda}(xI - M)$ with multiplicity at least $p$. Then
\begin{eqnarray*}
&&\operatorname{Imm}_{\lambda}(xI - \beta D(G) - \gamma A(G)) \\
&=& \operatorname{Imm}_{\lambda}(xI - (M + \beta I)) = \operatorname{Imm}_{\lambda}((x - 1)I - M)\\
&=&  c_{\lambda,0}(M)(x - 1)^n - c_{\lambda,1}(M)(x - 1)^{n-1} + c_{\lambda,2}(M)(x - 1)^{n-2} - \cdots + (-1)^n c_{\lambda,n}(M).
\end{eqnarray*}
This means that $\beta$ is a root of $\operatorname{Imm}_{\lambda}(xI - \beta D(G) - \gamma A(G))$ with multiplicity at least $p$.

\begin{theorem}\label{the2}
Let $G$ be a connected graph with $n$ vertices. The multiplicity of the root $\beta$ in $\operatorname{Imm}_{\lambda}(xI-\beta D(G) - \gamma A(G))$ is greater than or equal to the star degree of $G$.
\end{theorem}

\begin{proof}
Depending on the value of $p > 0$, we consider two cases.

\textbf{Case 1.} Suppose $p=0$. Note that the multiplicity of the root $\beta$ in $\operatorname{Imm}_{\lambda}(xI-\beta D(G) - \gamma A(G))$ is greater than or equal to zero, so we deduce that Theorem \ref{the2} holds.

\textbf{Case 2.} Suppose $p\geq1$. This means the graph contains $s$ star centers ($s \geq 1$), each with $\omega_1, \ldots, \omega_s$ pendant vertices ($\omega_i \geq 1$), and $\sum\limits_{i=1}^s (\omega_i - 1) = p$.
We can assign labels to vertices as follows: denote the $s$ center vertices as $v_n, v_{n-1}, \ldots, v_{n-s+1}$, denote the $\omega_1$ pendant vertices adjacent to $v_n$ as $v_1, \ldots, v_{\omega_1}$, denote the $\omega_2$ vertices adjacent to $v_{n-1}$ as $v_{\omega_1+1}, \ldots, v_{\omega_1+\omega_2}$, and so on; the remaining $n - s - \sum\limits_{i=1}^s \omega_i$ non-pendant vertices can be arbitrarily labeled as $v_{\omega_1+\cdots+\omega_r+1}, \ldots, v_{n-s}$.
Using this labeling scheme, we can assume the matrix $\beta D(G) + \gamma A(G)$ has the following form.
$$
\setlength{\arraycolsep}{4pt} 
\renewcommand{\arraystretch}{1.5} 
\footnotesize
\left[
\begin{array}{ccc|ccc|c|ccc|c|c|c|c|c}
\beta & \cdots & 0 & &  & &  & &  & &  &  &  &  & \gamma\\
0 & \ddots & 0 & & 0& & 0& & 0& & 0& 0& 0& 0& \vdots\\
0 & \cdots & \beta& &  & &  & &  & &  &  &  &  & \gamma\\
\hline
&  &  & \beta& \cdots& 0&  & &  & &  &  &  &  \gamma& \\
& 0&  &  & \ddots&  & 0& & 0& & 0& 0& 0& \vdots& 0\\
&  &  & 0& \cdots& \beta&  & &  & &  &  &  &  \gamma& \\
\hline
&  \vdots&  & & \vdots& &  \ddots& &  \vdots& &  \vdots&  &  {\scriptsize\rotatebox{90}{$\ddots$}}&  & \\
\hline
&  &  & &  &  &       & \beta&       & 0&  &      \gamma&  &   & \\
& 0&  & & 0&  & \cdots&  & \ddots&  & 0& \vdots& 0& 0& 0\\
&  &  & &  &  &       & 0&       & \beta&  &      \gamma&  &   & \\
\hline
& 0&  & & 0&  & 0&  & 0&  & *& *& *& *& *\\
\hline
& 0&  & & 0&  & 0&  \gamma& \cdots&  \gamma& *& \beta d(v_{n-s-1})&  &  & *\\
\hline
& 0&  & & 0&  & {\scriptsize\rotatebox{90}{$\ddots$}}&  & 0&   & *& &  \ddots&  &  \\
\hline
& 0&  & \gamma& \cdots&  \gamma& \cdots&  & 0&   & *& &  &  \beta d(v_{n-1})&  \\
\hline
\gamma& \cdots&  \gamma&  & 0&   & 0&  & 0&   & *& *&  &  &  \beta d(v_{n})\\
\end{array}
\right]_{n\times n}
$$
Using this notation, the matrix $\beta D(G) + \gamma A(G)$ satisfies $d(v_{j}) \geq 2$ for $j = \omega_1 + \omega_2 + \cdots + \omega_s + 1, \ldots, n$. Our goal is to prove that the multiplicity of the root $\beta$ in the polynomial $\operatorname{Imm}_{\lambda}(\beta D(G) + \gamma A(G) - xI)$ is at least $p$, which is equivalent to proving that the multiplicity of the root $0$ in the polynomial $\operatorname{Imm}_{\lambda}(\beta D(G) + \gamma A(G) - \beta I - xI)$ is also at least $p$.
Now consider the matrix $\beta D(G) + \gamma A(G) - \beta I$, which has the following block structure.
$$
\setlength{\arraycolsep}{3pt} 
\renewcommand{\arraystretch}{1} 
\footnotesize
\left[
\begin{array}{ccc|ccc|c|ccc|c|c|c|c|c}
 &  &  & &  & &  & &  & &  &  &  &  & \gamma\\
 & 0&  & & 0& & 0& & 0& & 0& 0& 0& 0& \vdots\\
 &  &  & &  & &  & &  & &  &  &  &  & \gamma\\
\hline
&  &  &  &   &  &  & &  & &  &  &  &  \gamma& \\
& 0&  &  &  0&  & 0& & 0& & 0& 0& 0& \vdots& 0\\
&  &  &  &   &  &  & &  & &  &  &  &  \gamma& \\
\hline
&  \vdots&  & & \vdots& &  \ddots& &  \vdots& &  \vdots&  &  {\scriptsize\rotatebox{90}{$\ddots$}}&  & \\
\hline
&  &  & &  &  &       &  &   &  &  &      \gamma&  &   & \\
& 0&  & & 0&  & \cdots&  &  0&  & 0& \vdots& 0& 0& 0\\
&  &  & &  &  &       &  &   &  &  &      \gamma&  &   & \\
\hline
& 0&  & & 0&  & 0&  & 0&  & *& *& *& *& *\\
\hline
& 0&  & & 0&  & 0&  \gamma& \cdots&  \gamma& *& \beta (d(v_{n-r-1})-1)&  &  & *\\
\hline
& 0&  & & 0&  & {\scriptsize\rotatebox{90}{$\ddots$}}&  & 0&   & *& &  \ddots&  &  \\
\hline
& 0&  & \gamma& \cdots&  \gamma& \cdots&  & 0&   & *& &  &  \beta(d(v_{n-1})-1)&  \\
\hline
\gamma& \cdots&  \gamma&  & 0&   & 0&  & 0&   & *& *&  &  &  \beta(d(v_{n})-1)\\
\end{array}
\right]_{n\times n}
$$
We aim to prove that $0$ is a root of $\operatorname{Imm}_{\lambda}(\beta D(G) + \gamma A(G) - \beta I - xI)$ with multiplicity at least $p$. Given $\omega_1+\cdots+\omega_s=\omega$ and $\sum\limits_{i=1}^s(\omega_i-1)=p$. This implies $\omega-s=p$. Now consider principal submatrices of order $n-(p-1)$. Each such principal submatrix contains a $[\omega-(p-1)]\times[(n-s)-(p-1)]$ zero submatrix in its upper left corner. Note that $(\omega-p+1)+(n-s-p+1)=n-p+2$. By Theorem \ref{the2.3}, we infer that the immanant of every principal submatrix of order $n-(p-1)$ of $\beta D(G) + \gamma A(G) - \beta I $ equals $0$. Indeed, each $n-(p-1)$ submatrix contains a $[\omega-(p-1)]\times[(n-s)-(p-1)]$ zero submatrix in its upper left corner. This indicates all these submatrices satisfy the conditions of Theorem \ref{the2.3}.
By Corollary \ref{cor3.1}, the immanant of each $z$ submatrix, where $n-(p-1)\leq z\leq n$, equals $0$. In particular, all principal submatrices of order $z$ in this range have zero immanant. Applying Equation (\ref{equ1.2}) and its conclusion, we deduce that $\beta$ is a root of $\operatorname{Imm}_{\lambda}(\beta D(G) + \gamma A(G)-xI)$ with multiplicity at least $p$.
\end{proof}

For $\beta=1$ and $\gamma=-1$, the above results directly yield the following conclusion.
\begin{corollary}\label{cor3.1}
Let $G$ be a connected graph with $n$ vertices. The multiplicity of the root $1$ in $\operatorname{Imm}_{\lambda}(xI-L(G))$ is greater than or equal to the star degree of $G$.
\end{corollary}

In the cases $\lambda = (1^n)$ and $\lambda = (n)$, the star degree of $G$ provides a lower bound for the multiplicity of the root $1$ of the characteristic polynomial and the permanental polynomial of $G$. This result was derived by Faria in \cite{far1}.

\begin{corollary}(Faria,\cite{far1})
Let $G$ be a connected graph with $n$ vertices. The multiplicity of the root $1$ in ${\rm det}(xI-L(G))$ is greater than or equal to the star degree of $G$.
\end{corollary}

\begin{corollary}(Faria,\cite{far2})
Let $G$ be a connected graph with $n$ vertices. The multiplicity of the root $1$ in ${\rm per}(xI-L(G))$ is greater than or equal to the star degree of $G$.
\end{corollary}

Merris \cite{mer2} further posed an open problem: whether the star degree of $G$ is always a lower bound for the multiplicity of the root $1$ in $d_{(2,1^{n-2})}(xI-L(G))$? Wu et al. \cite{wu2} proved that the multiplicity of the root $1$ of $\operatorname{Imm}_{(2,1^{n-2})}(xI-L(G))$ is at least the star degree of $G$, resolving this open problem. This result is contained in Theorem \ref{the2}.

\begin{corollary}(Wu et al., \cite{wu2})
Let $G$ be a connected graph. The multiplicity of the root $1$ in $\operatorname{Imm}_{(2,1^{n-2})}(xI-L(G))$ is greater than or equal to the star degree of $G$.
\end{corollary}

For $\beta=1$ and $\gamma=1$, the above results directly yield the following conclusion.
\begin{corollary}\label{cor3.1}
Let $G$ be a connected graph with $n$ vertices. The multiplicity of the root $1$ of $\operatorname{Imm}_{\lambda}(xI-Q(G))$ is at least the star degree of $G$.
\end{corollary}

For $0<\beta<1$ and $\gamma=1-\beta$, the above results directly yield the following conclusion.
\begin{corollary}\label{cor3.1}
Let $G$ be a connected graph with $n$ vertices. If $\alpha\notin\{0,1\}$, then the multiplicity of the root $\alpha$ of $\operatorname{Imm}_{\lambda}(xI-A_{\alpha}(G))$ is at least the star degree of $G$.
\end{corollary}

\section{Further discussion on theorem \ref{theorem4.1}}

In this section, we provide an in-depth discussion of Theorem \ref{theorem4.1}. Using Theorem \ref{theorem4.1}, we can obtain formulas for the first six coefficients of the hook immanantal polynomials of the Laplacian matrix, the signless Laplacian matrix, the adjacency matrix and the $A_{\alpha}$ matrix of a graph.

When $\beta=1$ and $\gamma=-1$, the first six coefficients of the hook immanantal polynomial of $L(G)$ can be obtained.
\begin{corollary}\label{cor5.2}
Let $G$ be a graph with $n$ vertices and $m$ edges, and $(d(v_{1}),d(v_{2}),\ldots,d(v_{n}))$ denote the degree sequence of $G$. Denote
\begin{eqnarray*}
\operatorname{Imm}_{(k,1^{n-k})}(xI -L(G) ) = \sum_{r=0}^{n} (-1)^r c_{(k,1^{n-k}),r}(L(G))x^{n-r}.
\end{eqnarray*}
Then
\begin{eqnarray*}
c_{(k,1^{n-k}),0}(L(G)) &=& \binom{n-1}{k-1}, \\
c_{(k,1^{n-k}),1}(L(G)) &=& F_1(G)\binom{n-1}{k-1} = 2m\binom{n-1}{k-1}, \\
c_{(k,1^{n-k}),2}(L(G)) &=& F_2(G)\binom{n-1}{k-1} +m\bigg[\binom{n-2}{k-2}-\binom{n-2}{k-1}\bigg]\quad (n \geq 3), \\
c_{(k,1^{n-k}),3}(L(G)) &=& F_3(G)\binom{n-1}{k-1} +\mathcal{M}_3^{1}(G)\bigg[\binom{n-2}{k-2}-\binom{n-2}{k-1}\bigg]\\
&&-2|\mathscr{C}_{3}(G)|\bigg[\binom{n-4}{k-4}-\binom{n-4}{k-1}\bigg]  \quad (n \geq 4),\\
c_{(k,1^{n-k}),4}(L(G)) &=& F_4(G)\binom{n-1}{k-1} +\mathcal{M}_4^{1}(G)\bigg[\binom{n-2}{k-2}-\binom{n-2}{k-1}\bigg]\\
&&-2\mathcal{C}_{4}^{3}(G)\bigg[\binom{n-4}{k-4}-\binom{n-4}{k-1}\bigg]+2|\mathscr{C}_{4}(G)|\bigg[\binom{n-5}{k-5}-\binom{n-5}{k-1}\bigg] \\
&&+\bigg[\binom{m}{2}-\sum\limits_{i=1}^{n} \binom{d(v_{i})}{2}\bigg]\bigg[\binom{n-5}{k-5}-2\binom{n-5}{k-3}+\binom{n-5}{k-1}\bigg] \quad (n \geq 5),
\end{eqnarray*}
and
\begin{eqnarray*}
c_{(k,1^{n-k}),5}(L(G)) &=& F_{5}(G)\binom{n-1}{k-1} +\mathcal{M}_5^{1}(G)\bigg[\binom{n-2}{k-2}-\binom{n-2}{k-1}\bigg]\\
&&-2\mathcal{C}_{5}^{3}(G)\bigg[\binom{n-4}{k-4}+\binom{n-4}{k-1}\bigg]+2\mathcal{C}_{5}^{4}(G)\bigg[\binom{n-5}{k-5}-\binom{n-5}{k-1}\bigg] \\
&&+\mathcal{M}_{5}^{2}(G)\bigg[\binom{n-5}{k-5}-2\binom{n-5}{k-3}+\binom{n-5}{k-1}\bigg]- 2|\mathscr{C}_{5}(G)|\bigg[\binom{n-6}{k-6}-\binom{n-6}{k-1}\bigg]\\
&&-\sum\limits_{j=1}^{|\mathscr{C}_{3}(G)|} 2(m+3-\mathscr{T}_{j}(G))\bigg[\binom{n-6}{k-6}+\binom{n-6}{k-3}-\binom{n-6}{k-4}-\binom{n-6}{k-1}\bigg] \quad (n \geq 6).
\end{eqnarray*}
\end{corollary}
Corollary \ref{cor5.2} directly yields the four coefficients of the second immanantal polynomials of  graphs introduced by Merris \cite{mer4}.

\begin{corollary}(Merris, \cite{mer4})\label{cor5.3}
Let $G$ be a graph with $n$ vertices and $m$ edges. Denote
\begin{eqnarray*}
\operatorname{Imm}_{(2,1^{n-2})}(xI -L(G) ) = \sum_{r=0}^{n} (-1)^r c_{2,r}(L(G))x^{n-r}.
\end{eqnarray*}
Then
\begin{eqnarray*}
c_{(2,1^{n-2}),0}(L(G)) &=& n-1, \\
c_{(2,1^{n-2}),1}(L(G)) &=& (n-1)F_1(G) = 2m(n-1), \\
c_{(2,1^{n-2}),2}(L(G)) &=& (n-1)F_2(G)-m(n-3)\quad (n \geq 3),
\end{eqnarray*}
and
\begin{eqnarray*}
c_{(2,1^{n-2}),3}(L(G)) &=& (n-1)F_3(G) -(n-3)\mathcal{M}_3^{1}(G)-2(n-4)|\mathscr{C}_{3}(G)|  \quad (n \geq 4).
\end{eqnarray*}
\end{corollary}

Under the conditions $\beta=1$ and $\gamma=1$, the first six coefficients of the hook immanantal polynomial of $Q(G)$ can be derived.

\begin{corollary}\label{cor5.4}
Let $G$ be a graph with $n$ vertices and $m$ edges, and let $(d(v_{1}),d(v_{2}),\ldots,d(v_{n}))$ denote the degree sequence of $G$. Denote
\begin{eqnarray*}
\operatorname{Imm}_{(k,1^{n-k})}(xI -Q(G) ) = \sum_{r=0}^{n} (-1)^r c_{k(k,1^{n-k}),r}(Q(G))x^{n-r}.
\end{eqnarray*}
Then
\begin{eqnarray*}
c_{(k,1^{n-k}),0}(Q(G)) &=& \binom{n-1}{k-1}, \\
c_{(k,1^{n-k}),1}(Q(G)) &=& F_1(G)\binom{n-1}{k-1} = 2m\binom{n-1}{k-1}, \\
c_{(k,1^{n-k}),2}(Q(G)) &=& F_2(G)\binom{n-1}{k-1} +m\bigg[\binom{n-2}{k-2}-\binom{n-2}{k-1}\bigg]\quad (n \geq 3), \\
c_{(k,1^{n-k}),3}(Q(G)) &=& F_3(G)\binom{n-1}{k-1} +\mathcal{M}_3^{1}(G)\bigg[\binom{n-2}{k-2}-\binom{n-2}{k-1}\bigg]\\
&&+2|\mathscr{C}_{3}(G)|\bigg[\binom{n-4}{k-4}-\binom{n-4}{k-1}\bigg]  \quad (n \geq 4),\\
c_{(k,1^{n-k}),4}(Q(G)) &=& F_4(G)\binom{n-1}{k-1} +\mathcal{M}_4^{1}(G)\bigg[\binom{n-2}{k-2}-\binom{n-2}{k-1}\bigg]\\
&&+2\mathcal{C}_{4}^{3}(G)\bigg[\binom{n-4}{k-4}-\binom{n-4}{k-1}\bigg]+2|\mathscr{C}_{4}(G)|\bigg[\binom{n-5}{k-5}-\binom{n-5}{k-1}\bigg] \\
&&+\bigg[\binom{m}{2}-\sum\limits_{i=1}^{n} \binom{d(v_{i})}{2}\bigg]\bigg[\binom{n-5}{k-5}-2\binom{n-5}{k-3}+\binom{n-5}{k-1}\bigg] \quad (n \geq 5),
\end{eqnarray*}
and
\begin{eqnarray*}
c_{(k,1^{n-k}),5}(Q(G)) &=& F_{5}(G)\binom{n-1}{k-1} +\mathcal{M}_5^{1}(G)\bigg[\binom{n-2}{k-2}-\binom{n-2}{k-1}\bigg]\\
&&+2\mathcal{C}_{5}^{3}(G)\bigg[\binom{n-4}{k-4}+\binom{n-4}{k-1}\bigg]+2\mathcal{C}_{5}^{4}(G)\bigg[\binom{n-5}{k-5}-\binom{n-5}{k-1}\bigg] \\
&&+\mathcal{M}_{5}^{2}(G)\bigg[\binom{n-5}{k-5}-2\binom{n-5}{k-3}+\binom{n-5}{k-1}\bigg]+ 2|\mathscr{C}_{5}(G)|\bigg[\binom{n-6}{k-6}-\binom{n-6}{k-1}\bigg]\\
&&+\sum\limits_{j=1}^{|\mathscr{C}_{3}(G)|} 2(m+3-\mathscr{T}_{j}(G))\bigg[\binom{n-6}{k-6}+\binom{n-6}{k-3}-\binom{n-6}{k-4}-\binom{n-6}{k-1}\bigg] \quad (n \geq 6).
\end{eqnarray*}
\end{corollary}

Setting $\beta=0$ and $\gamma=1$ gives the first six coefficients of the hook immanantal polynomial of $A(G)$.
\begin{corollary}\label{cor5.5}
Let $G$ be a graph with $n$ vertices and $m$ edges, and $(d(v_{1}),d(v_{2}),\ldots,d(v_{n}))$ denote the degree sequence of $G$. Denote
\begin{eqnarray*}
\operatorname{Imm}_{(k,1^{n-k})}(xI -A(G) ) = \sum_{r=0}^{n} (-1)^r c_{(k,1^{n-k}),r}(A(G))x^{n-r}.
\end{eqnarray*}
Then
\begin{eqnarray*}
c_{(k,1^{n-k}),0}(A(G)) &=& \binom{n-1}{k-1}, \\
c_{(k,1^{n-k}),1}(A(G)) &=& 0, \\
c_{(k,1^{n-k}),2}(A(G)) &=& m\bigg[\binom{n-2}{k-2}-\binom{n-2}{k-1}\bigg]\quad (n \geq 3), \\
c_{(k,1^{n-k}),3}(A(G)) &=& 2|\mathscr{C}_{3}(G)|\bigg[\binom{n-4}{k-4}-\binom{n-4}{k-1}\bigg]  \quad (n \geq 4),\\
c_{(k,1^{n-k}),4}(A(G)) &=&\bigg[\binom{m}{2}-\sum\limits_{i=1}^{n} \binom{d(v_{i})}{2}\bigg]\bigg[\binom{n-5}{k-5}-2\binom{n-5}{k-3}+\binom{n-5}{k-1}\bigg] \\
&&+ 2|\mathscr{C}_{4}(G)|\bigg[\binom{n-5}{k-5}-\binom{n-5}{k-1}\bigg]\quad (n \geq 5),
\end{eqnarray*}
and
\begin{eqnarray*}
c_{(k,1^{n-k}),5}(A(G)) &=&\sum\limits_{j=1}^{|\mathscr{C}_{3}(G)|} 2(m+3-\mathscr{T}_{j}(G))\bigg[\binom{n-6}{k-6}+\binom{n-6}{k-3}-\binom{n-6}{k-4}-\binom{n-6}{k-1}\bigg]  \\
&&+ 2|\mathscr{C}_{5}(G)|\bigg[\binom{n-6}{k-6}-\binom{n-6}{k-1}\bigg]\quad (n \geq 6).
\end{eqnarray*}
\end{corollary}

Using Corollary \ref{cor5.5}, we can derive the first six coefficients of the permanent polynomial of $G$, which were proposed by Liu and Zhang \cite{liu}.
\begin{corollary}( Liu and Zhang, \cite{liu})\label{cor5.6}
Let $G$ be a graph with $n$ vertices and $m$ edges, and $(d(v_{1}),d(v_{2}),\ldots,d(v_{n}))$ denote the degree sequence of $G$. Denote
\begin{eqnarray*}
{\rm per}(xI -A(G) ) = \sum_{r=0}^{n} (-1)^r c_{(n),r}(A(G))x^{n-r}.
\end{eqnarray*}
Then
\begin{eqnarray*}
c_{(n),0}(A(G)) &=& 1, \\
c_{(n),1}(A(G)) &=& 0, \\
c_{(n),2}(A(G)) &=& m, \\
c_{(n),3}(A(G)) &=& 2|\mathscr{C}_{3}(G)|,\\
c_{(n),4}(A(G)) &=&\binom{m}{2}-\sum\limits_{i=1}^{n} \binom{d(v_{i})}{2}+ 2|\mathscr{C}_{4}(G)|.
\end{eqnarray*}
\end{corollary}

Under the conditions $0\leq\beta\leq1$ and $\gamma=1-\beta$, the first six coefficients of the hook immanantal polynomial of $A_{\alpha}(G)$ can be derived.

\begin{corollary}\label{thm4.1}
Let $G$ be a graph with $n$ vertices and $m$ edges, and let $(d(v_{1}),d(v_{2}),\ldots,d(v_{n}))$ denote the degree sequence of $G$. Denote
\begin{eqnarray*}
\operatorname{Imm}_{(k,1^{n-k})}(xI - A_{\alpha}(G)) = \sum_{r=0}^{n} (-1)^r c_{(k,1^{n-k}),r}(A_{\alpha}(G))x^{n-r}.
\end{eqnarray*}
Then
\begin{eqnarray*}
c_{(k,1^{n-k}),0}(A_{\alpha}(G)) &=& \binom{n-1}{k-1}, \\
c_{(k,1^{n-k}),1}(A_{\alpha}(G)) &=& 2m\alpha\binom{n-1}{k-1}, \\
c_{(k,1^{n-k}),2}(A_{\alpha}(G)) &=& \alpha^{2}F_2(G)\binom{n-1}{k-1} +m(1-\alpha)^{2}\bigg[\binom{n-2}{k-2}-\binom{n-2}{k-1}\bigg]\quad (n \geq 3), \\
c_{(k,1^{n-k}),3}(A_{\alpha}(G)) &=& \alpha^{3}F_3(G)\binom{n-1}{k-1} +\alpha(1-\alpha)^{2}\mathcal{M}_3^{1}(G)\bigg[\binom{n-2}{k-2}-\binom{n-2}{k-1}\bigg]\\
&&+2|\mathscr{C}_{3}(G)|(1-\alpha)^{3}\bigg[\binom{n-4}{k-4}-\binom{n-4}{k-1}\bigg]  \quad (n \geq 4),
\end{eqnarray*}
\begin{eqnarray*}
&&c_{(k,1^{n-k}),4}(A_{\alpha}(G)) \\
&=& \alpha^{4}F_4(G)\binom{n-1}{k-1} +\alpha^{2}(1-\alpha)^{2}\mathcal{M}_4^{1}(G)\bigg[\binom{n-2}{k-2}-\binom{n-2}{k-1}\bigg]\\
&&+2\alpha(1-\alpha)^{3}\mathcal{C}_{4}^{3}(G)\bigg[\binom{n-4}{k-4}-\binom{n-4}{k-1}\bigg]+2|\mathscr{C}_{4}(G)|(1-\alpha)^{4}\bigg[\binom{n-5}{k-5}-\binom{n-5}{k-1}\bigg] \\
&&+\bigg[\binom{m}{2}-\sum\limits_{i=1}^{n} \binom{d(v_{i})}{2}\bigg]\bigg[\binom{n-5}{k-5}-2\binom{n-5}{k-3}+\binom{n-5}{k-1}\bigg](1-\alpha)^{4} \quad (n \geq 5),
\end{eqnarray*}
and
\begin{eqnarray*}
&&c_{(k,1^{n-k}),5}(A_{\alpha}(G))\\
 &=& \alpha^{5}F_{5}(G)\binom{n-1}{k-1} +\alpha^{3}\gamma^{2}\mathcal{M}_5^{1}(G)\bigg[\binom{n-2}{k-2}-\binom{n-2}{k-1}\bigg]\\
&&+2\alpha^{2}(1-\alpha)^{3}\mathcal{C}_{5}^{3}(G)\bigg[\binom{n-4}{k-4}+\binom{n-4}{k-1}\bigg]+2\alpha(1-\alpha)^{4}\mathcal{C}_{5}^{4}(G)\bigg[\binom{n-5}{k-5}-\binom{n-5}{k-1}\bigg] \\
&&+\alpha(1-\alpha)^{4}\mathcal{M}_{5}^{2}(G)\bigg[\binom{n-5}{k-5}-2\binom{n-5}{k-3}+\binom{n-5}{k-1}\bigg]+ 2|\mathscr{C}_{5}(G)|(1-\alpha)^{5}\bigg[\binom{n-6}{k-6}-\binom{n-6}{k-1}\bigg]\\
&&+\sum\limits_{j=1}^{|\mathscr{C}_{3}(G)|} 2(m+3-\mathscr{T}_{j}(G))(1-\alpha)^{5}\bigg[\binom{n-6}{k-6}+\binom{n-6}{k-3}-\binom{n-6}{k-4}-\binom{n-6}{k-1}\bigg] \quad (n \geq 6).
\end{eqnarray*}
\end{corollary}

\section{Conclusion}
This paper mainly studies the immanant polynomials of the linear combination matrix $\beta D(G) + \gamma A(G)$ of a graph. We establish upper and lower bounds for the coefficients of the immanantal polynomials of the linear combination matrices of graphs. These bounds not only extend the known inequalities of Chan and Lam from trees and bipartite graphs to general graphs, but also cover special cases such as the Laplacian matrix, the signless Laplacian matrix, and the $A_\alpha$ matrix.

By precisely characterizing the first six coefficients of the hook immanantal polynomial, we derive a necessary and sufficient condition for the equality of the hook immanantal polynomials of linear combination matrices of two regular graphs. Moreover, necessary and sufficient conditions for the equality of the hook immanantal polynomials of the Laplacian matrix, the signless Laplacian matrix, and the $A_\alpha$ matrix of regular graphs are obtained. In addition, we provide explicit expressions for specific cases, including the coefficients of Merris's second immanantal polynomial and the coefficients of Liu and Zhang's permanent polynomial.

Through extending the matrix versions of the Frobenius--K\"onig theorem and the Laplace expansion theorem to immanant functions, we show that the star degree of a graph equals the multiplicity of the root $\beta$ in the immanantal polynomial of its linear combination matrix. This conclusion subsumes the result of Wu et al. on the root multiplicity of the second immanant of the Laplacian matrix as a special case.

Based on general theoretical principles, this paper systematically deduces many previously known important conclusions, achieving the generalization and unification of multiple classical results. This not only verifies the validity of existing findings,  but also helps people understand the rich content and inherent patterns of the generalized matrix function immanant in graph theory from a unified perspective.

\noindent{\bf Data Availability}\\
No data were used to support this study.

\noindent{\bf Disclosure statement}\\
The authors declare that they have no conflicts of interest.

\noindent{\bf Funding}\\
This research is supported by NSFC (No. 12261071)  and NSF of Qinghai Province (No. 2025-ZJ-902T).


\begin{thebibliography}{abcdsfgh}



\bibitem{agr}M. Agrawal, Determinant versus permanent,   European Mathematical Society, 2007.

\bibitem{bol} D. Bolognini, P. Sentinelli, Immanant varieties, \textit{Linear Algebra Appl.} 682 (2024) 164--190.


\bibitem{bot} P. Botti, R. Merris, Almost all trees share a complete set of immanantal polynomials, \textit{J. Graph Theory} 17 (1993) 467--476.


\bibitem{bro} A.E. Brouwer, W.H. Haemers, Spectra of graphs, \textit{Springer}, 2011.

\bibitem{bur} P. B\"{u}rgisser, The computational complexity of immanants, \textit{SIAM J. Computing} 30 (2000) 1023--1040.

\bibitem{cas} G. Cash, Immanants and immanantal polynomials of chemical graphs, \textit{J. Chem. Inform. Comput. Sci.} 43 (2003) 1942--1946.

\bibitem{chan2} O. Chan, T. Lam, Immanant inequalities for Laplacians of trees, \textit{SIAM J. Matrix Anal. Appl.} 21 (1999) 129--144.

\bibitem{chan3} O. Chan, T. Lam, Hook immanantal inequalities for Laplacians of trees, \textit{Linear Algebra Appl.} 261 (1997) 23--47.

\bibitem{chan4} O. Chan, T. Lam, Vertex orientations and immanants of  bipartite graphs, \textit{Available on Semantic Scholar},  1997.

\bibitem{chan5} O. Chan, T. Lam, Immanant of Laplacian matrix of trees, manuscript, 1997.


\bibitem{cvet} D.M. Cvetkovi\'{c}, P. Rowlinson, S. Simi\'{c}, An introduction to the theory of graph spectra, \textit{ Cambridge: Cambridge university press.}, 2010.

\bibitem{deg} H. de Guise, D. Spivak, J. Kulp, and I. Dhand, D-functions and immanants of unitary matrices and submatrices, \textit{J. Phys. A: Math. Theor.} 49 (2016) 09LT01 (12pp).


\bibitem{dia} P. Diaconis and S. N. Evans, Immanants and finite point processes, \textit{J. Combinat. Theory, Ser. A} 91 (2000) 305--321.

\bibitem{don} X. Dong, T. Wu, H. Lai, Some immanantal inequalities and equalities for linear combination matrices of (di)graphs, submitted.

\bibitem{far1} I. Faria, Multiplicity of integer roots of polynomials of graphs, \textit{Linear Algebra Appl.} 229 (1995) 15--35.

\bibitem{far2} I. Faria, Permanental roots and the star degree of a graph, \textit{Linear Algebra Appl.} 64 (1985) 255--265.



\bibitem{gly} D.G. Glynn,  The permanent of a square matrix, \textit{European J. Combin.} 31 (2010) 1887--1891.

\bibitem{god} C.D. Godsil, Algebraic combinatorics, \textit{Chapman and Hall, New York}, 1993.



\bibitem{gou2} I. Goulden, D. Jackson, Immanants of combinatorial matrices, \textit{J. Algebra} 148 (1992) 305--324.



\bibitem{hai} M. Haiman, Hecke algebra characters and immanant conjectures, \textit{J. Amer. Math. Soc.} 6 (1993) 569--595.



\bibitem{jam} C. James, A. Kerber, The representation theory of the symmetric group, Encyclopedia of Mathematics, Addison-Wesley, Reading, Mass., 1981.

\bibitem{li1} C. Li, A. Zaharia, Induced operators on symmetry classes of tensors, \textit{Trans. Am. Math. Soc.} 354 (2002) 807--836.


\bibitem{liu} S. Liu, H. Zhang, On the characterizing properties of the permanental polynomials of graphs, \textit{Linear Algebra Appl.} 438 (2013) 157--172.

\bibitem{liw} W. Li, S. Liu, T. Wu, H. Zhang, On the permanental polynomials of graphs, in: Graph Polynomials, Y. Shi, M. Dehmer, X. Li, I. Gutman (Eds.), \textit{CRC Press, Boca Raton}, 2017, 101--122.


\bibitem{lit}
D.E. Littlewood, The theory of group characters, 2nd ed., Oxford Univ. Press, London, 1950.


\bibitem{mar} M. Marcus, H. Minc, Generalized matrix functions, \textit{Trans. Am. Math. Soc.} 116 (1965) 316--329.

\bibitem{mar2} M. Marcus, P.J. Nikolai, Inequalities for some monotone matrix functions, \textit{Canad. J. Math.} 21 (1969) 485-194.

\bibitem{mer1} R. Merris, Single-hook characters and Hamiltonian circuits, \textit{Linear Multilinear Algebra} 14 (1983) 21--35.


\bibitem{mer4} R. Merris, The second immanantal polynomial and the centroid of a graph, \textit{SIAM J. Algebraic Discrete Methods} 7 (1986) 484--503.


\bibitem{mer2} R. Merris, Immanantal invariants of graphs, \textit{Linear Algebra Appl.} 401 (2005) 67--75.

\bibitem{mer3} R. Merris, K.R. Rebman, W. Watkins, Permanental polynomials of graphs, \textit{Linear Algebra Appl.}  38 (1981) 273--288.

\bibitem{min} H. Minc, Permanents, Addision-Wesley, London, 1978.

\bibitem{mow} A. Mowshowitz, The characteristic polynomial of a graph, \textit{J. Combinator. Theory, Ser. B}  12  (1972)  177--193.




\bibitem{rho} B. Rhoades, M. Skandera, Kazhdan--Lusztig immanants and products of matrix minors, \textit{J. Algebra} 304 (2006) 793--811.

\bibitem{sag} B. Sagan, The symmetric group: representations, combinatorial algorithms, and symmetric functions, \textit{Wadsworth}, 1991.

\bibitem{sch} I. Schur, \"{U}ber endliche Gruppen undHermitesche Formen, \textit{Math. Z.} 1 (1918), 184--207.

\bibitem{sta} R. P. Stanley, J. R. Stembridge, On immanants of Jacobi--Trudi matrices and permutations with restricted position, \textit{J. Combinator. Theory, Ser. A} 62 (1993) 261--279.

\bibitem{ste} J. Stembridge, Some conjectures for immanant, \textit{Can. J. Math.} 44 (1992) 1079--1099.




\bibitem{van}  P. Van Mieghem,  Graph spectra for complex networks, Cambridge university press, 2023.


\bibitem{wu2} T. Wu, Y. Yu, L. Feng, X. Gao, On the second immanantal polynomials of graphs, \textit{Discrete Math.} 347 (2024) 114105.



\bibitem{yu} G. Yu, H. Qu, The coefficients of the immanantal polynomial, \textit{Appl. Math. Comput.} 339 (2018) 38--44.






\end{thebibliography}
\end{document}